\documentclass[11pt]{amsart}
\usepackage{amssymb}
\usepackage[hyphens]{url}

\newtheorem{theorem}{Theorem}[section]
\newtheorem{corollary}[theorem]{Corollary}
\newtheorem{lemma}[theorem]{Lemma}
\newtheorem{proposition}[theorem]{Proposition}

\newtheorem{problem}{Problem}

\theoremstyle{definition}
\newtheorem{definition}[theorem]{Definition}

\newtheorem{question}[theorem]{Question}

\theoremstyle{remark}

\newtheorem{example}[theorem]{Example}

\numberwithin{equation}{section}

\def\Z {{\mathbb{Z}}}

\def\3{{|\!|\!|}}

\def\rank{{\mathrm{rank}}}
\def\lh{{\mathrm{lh}}}

\def\dgnb{{\mathrm{dgnb}}}
\renewcommand{\restriction}{\mathord{\upharpoonright}}
\def\gnb#1{G_{\langle #1\rangle}}
\def\hier#1{$#1$-CLI}
\def\otr#1#2{T_{\mathcal{#1}}^{#2}}
\def\aut{{\mathrm{Aut}}}

\def\dsZ{{\Z^*}}
\def\rk{\mathrm{rk}}

\def\itx#1{{\mbox{\textrm{ #1 }}}}
\def\itxr#1{{\mbox{\textrm{ #1}}}}
\def\itxl#1{{\mbox{\textrm{#1 }}}}

\begin{document}

\title[On the complexity of extensions of non-archimedean CLI groups]{On the complexity of extensions of non-archimedean Polish groups admitting a compatible complete left-invariant metric}
\author{Longyun Ding}
\address{School of Mathematical Sciences and LPMC, Nankai University, Tianjin, 300071, P.R.China}
\email{dingly@nankai.edu.cn}
\thanks{Research is partially supported by the National Natural Science Foundation of China (Grant No. 12271264).}
\author{Xu Wang}
\email{xwang@mail.nankai.edu.cn}

\subjclass[2010]{Primary 03E15, 22A05}
\keywords{hierarchy, non-archimedean Polish group, group extension}


\begin{abstract}
    In this article, motivated by a problem asked by Allison and Panagiotopoulos~\cite{SAllisonAPana}, we study a problem concerning the complexity of group extensions within a hierarchy (denoted by \hier\alpha\space and L-\hier\alpha) on the class of non-archimedean CLI Polish groups:

    Given a non-archimedean Polish group $G$ and one of its closed normal subgroup $N$, suppose $N$ and $G/N$ are \hier\alpha\space and \hier\beta, respectively. Is $G$ always \hier{(\alpha+\beta)}?

    We provide a positive answer under a certain additional assumption. We then construct two examples yielding negative answers:
    \begin{enumerate}
        \item for each countably infinite ordinal $\alpha$, there exists a group $G$ that is not \hier\alpha, but $G$ has a \hier1\space normal subgroup $N$ such that $G/N$ is proper \hier\alpha;
        \item there exists a proper \hier3\space group $U$ that has an abelian normal subgroup $N$ such that $U/N$ is also abelian.
    \end{enumerate}
    These examples also provide negative answers to the original problem raised by Allison and Panagiotopoulos.

    Finally, we show that if $N$ and $G/N$ are \hier\alpha\space and \hier\beta\space with $\beta>0$, respectively, then $G$ is \hier{\beta\cdot(\omega\cdot\alpha+1)}, which gives an upper bound on the complexity of the extended group.





\end{abstract}
\maketitle

\section{Introduction}

A {\it Polish} group is a topological group with the underlying topology Polish. A {\it non-archimedean} Polish group is a Polish group that has an open neighborhood basis of the identity element consisting of open subgroups. A theorem of Becker and Kechris (cf. \cite[Theorem 1.5.1]{BK}) states that a Polish group $G$ is non-archimedean iff $G$ can be embedded as a closed subgroup of $S_\infty$, i.e., the permutation group on $\omega$. A metric $d$ on a group $G$ is said to be {\it left-invariant} if $d(gh,gk)=d(h,k)$ for any $g,h,k\in G$. A CLI group is a topological group that admits a compatible complete left-invariant metric.

Malicki \cite{malicki11} introduced the notion of the orbit tree for closed subgroups of $S_\infty$.
He proved that a closed subgroup of $S_\infty$ is CLI iff its corresponding orbit tree is well-founded.
Moreover, he showed the class of all CLI Polish groups is $\pmb{\Pi}_1^1$ non-Borel.
Subsequently, Xuan \cite{xuan} modified the definition of the orbit tree, and further discussed the relationship between the rank of an orbit tree and topological aspects of a non-archimedean Polish group. He showed that a non-archimedean Polish group is locally compact iff its corresponding orbit tree has finite rank.
Later, Ding and Wang \cite{DingW} also modified the definition of the orbit tree, by which they introduced a hierarchy on the class of all non-archimedean CLI Polish groups, denoted by \hier\alpha\space and L-\hier\alpha.
Allison and Panagiotopoulos \cite{SAllisonAPana} defined the notion of $\alpha$-balancedness for general Polish groups.
They used it to give a hierarchy that stratifies the class of all CLI Polish groups. In detail, given a Polish group $G$, they defined $\rk(V,U)$ for $V,U\subseteq_1 G$, where $V\subseteq_1 G$ indicates that $V$ is an open neighborhood of identity element of $G$. By letting $\rk(G)=\sup\{\rk(V,G)+1:V\subseteq_1 G\}$, they proved that $G$ is CLI iff $\rk(G)$ is a countable ordinal. For any ordinal $\alpha$, $G$ is said to be $\alpha$-balanced if $\rk(G)\le\alpha$. As an application of the notion of $\alpha$-balancedness, they further proved that the class of all CLI Polish groups forms a $\pmb{\Pi}_1^1$-complete set.

A well-known result states that for a Polish group $G$ and a closed normal subgroup $N$ of $G$, $G$ is CLI iff both $N$ and $G/N$ are CLI (cf. \cite[Theorem 2.2]{gao98}). Allison and Panagiotopoulos in \cite{SAllisonAPana} raised the following problem:

\begin{problem}
    Let $H$ be a closed normal subgroup of a topological group $G$. Is it true that
    \begin{equation*}
        \rk(G)\leq\sup\{\rk(V,H;H)+\rk(G/H):V\subseteq_1H\}?
    \end{equation*}
Here $\rk(V,H;H)$ means that the underlying topological group for function $\rk(V,H)$ is $H$.
\end{problem}

Instead of solving this problem directly, we ask a similar problem:

\begin{problem}
    Let $N$ be a closed normal subgroup of a non-archimedean Polish group $G$. Suppose $N$ is \hier\alpha\space and $G/N$ is \hier\beta, where $\alpha,\beta<\omega_1$. Is $G$ \hier{(\alpha+\beta)}?
\end{problem}

To build a connection between these two problems, we first discuss the relationship between a non-archimedean Polish group being $\alpha$-balanced and being \hier\alpha. We show that

\begin{theorem}
    Let $G$ be a non-archimedean Polish group and $\alpha<\omega_1$. Then $G$ is \hier\alpha\space iff $G$ is $(1+\alpha)$-balanced.
\end{theorem}

In view of this theorem, a positive answer to Problem 1 would yield a positive answer to Problem 2. In case of the semidirect products, we obtain a positive answer to Problem 2 under a certain assumption (see Theorem~\ref{positive answer for semidirect product}). We then apply this theorem to two classes of examples: the wreath products $G\wr H$ and the semidirect products $G\ltimes\Z^\omega$.

However, these results cannot be generalized to all non-archimedean CLI Polish groups. In fact, the following result allows us to construct a counterexample to Problem 2:

\begin{theorem}\label{introduction counterexample 1}
    Let $\Gamma$ be a countably infinite group, and let $G$ be a closed subgroup of $\aut(\Gamma)$. If $G$ is proper \hier\alpha, then $A=G\ltimes_\rho\Gamma$ is proper L-\hier\alpha, where $\rho:G\to\aut(\Gamma)$ is given by
    $\rho(g)(\gamma)=g(\gamma)$ for $g\in G$ and $\gamma\in\Gamma$.
\end{theorem}

Consequently, we may choose a proper \hier\alpha\space group $G$ for some $\alpha\geq\omega$, and define $A=G\ltimes_\rho\Gamma$. Then $A$ is proper L-\hier\alpha, and hence not \hier{(1+\alpha)}.

In~\cite{DingW}, we define the notion $\rank(G)$ so that $\rank(G)\le\alpha$ iff $G$ is L-\hier\alpha\space. Note that in Theorem~\ref{introduction counterexample 1} we have $\rank(A)=\rank(G)=\alpha$ and $\rank(\Gamma)=0$. One may expect that
$$\rank(G)\leq\rank(N)+\rank(G/N)$$
holds for all non-archimedean Polish groups $G$ and their closed normal subgroups $N$. But this is refuted by the following more fundamental counterexample:


\begin{theorem}\label{rank counterexample}
    There exists a non-archimedean Polish group $U$ that is proper \hier3, although $U$ has a closed normal subgroup $N$ such that both $N$ and $U/N$ are abelian, thus \hier1.
\end{theorem}

Surprisingly, despite the negative answer to Problem 2, we still managed to find an upper bound for the complexity of group $G$.

\begin{theorem}
    Let $G$ be a non-archimedean Polish group, $N$ be a proper closed normal subgroup of $G$. If $N$ is \hier\alpha\space and $G/N$ is \hier\beta, then $G$ is \hier{\beta\cdot(\omega\cdot\alpha+1)}.
\end{theorem}

Note that this upper bound is quite large. For instance, even if both $N$ and $G/N$ are \hier{1}, the upper bound in the above theorem for $G$ is \hier{(\omega+1)}.
Therefore, we would like to know whether the upper bound \hier{\beta\cdot(\omega\cdot\alpha+1)} given in the above theorem is optimal.

\setcounter{problem}{0}

\section{Preliminaries}

Let $\alpha$ be an ordinal. We define
$$\omega(\alpha)=\max\{0,\lambda:\lambda\le\alpha\mbox{ is a limit ordinal}\}.$$
Then $\alpha=\omega(\alpha)+m$ for some $m\in\omega$.

Let $E$ be an equivalence relation on a set $X$, $x\in X$, and $A\subseteq X$. The $E$-equivalence class of $x$ is $[x]_E=\{y\in X:xEy\}$. Similarly, the $E$-saturation of $A$ is $[A]_E=\{y\in X:\exists z\in A\,(yEz)\}$.

Let $G$ be a group. We denote the identity element of $G$ by $1_G$. For a subgroup $H$ of $G$, we denote by $G/H$ the set of all left-cosets of $H$.

A topological space is {\it Polish} if it is separable and completely metrizable. A topological group is {\it Polish} if its underlying topology is Polish. Let $G$ be a Polish group and $X$ a Polish space. An action of $G$ on $X$, denoted by $G\curvearrowright X$, is a map $a:G\times X\to X$ that satisfies $a(1_G,x)=x$ and $a(gh,x)=a(g,a(h,x))$ for $g,h\in G$ and $x\in X$. The pair $(X,a)$ is called a {\it Polish $G$-space} if $a$ is continuous. When the action is understood from context, we write $g\cdot x$ instead of $a(g,x)$.
The {\it orbit equivalence relation} $E_G^X$ is defined as $xE_G^Xy\iff\exists g\in G\,(g\cdot x=y)$. Note that the $E_G^X$-equivalence class of $x$ is $G\cdot x=[x]_G=\{g\cdot x:g\in G\}$, which is also called the {\it $G$-orbit} of $x$. The class of all $G$-orbits is $X/G=\{[x]_G:x\in X\}$. Similarly, for $A\subseteq X$, the $E_G^X$-saturation of $A$ is $G\cdot A=[A]_G=\{g\cdot x:g\in G\land x\in A\}$.

Let $E,F$ be two equivalence relations on sets $X,Y$, respectively. A {\it homomorphism} from $E$ to $F$ is a function $f:X\to Y$ satisfying
\begin{equation*}
    \forall\,x,y\in X\left(x\,E\,y\implies f(x)\,F\,f(y)\right);
\end{equation*}
and a {\it reduction} from $E$ to $F$ is a function $f:X\to Y$ satisfying
\begin{equation*}
    \forall\,x,y\in X\left(x\,E\,y\Longleftrightarrow f(x)\,F\,f(y)\right).
\end{equation*}
In particular, if there exists a reduction from $E$ to $F$ which is a bijection, we say that $E$ is {\it isomorphic} to $F$.

Let $<$ be a binary relation on a set $T$. We say that $(T,<)$ is a {\it tree} if
\begin{enumerate}
\item[(1)] $\forall s\in T\,(s\not<s)$,
\item[(2)] $\forall s,t,u\in T\,((s<t\land t<u)\implies s<u)$,
\item[(3)] $\forall s\in T\,(|\{t\in T:t<s\}|<\omega\land\forall\,t,u<s\,(t=u\lor t<u\lor u<t))$.
\end{enumerate}
For $s\in T$, we define the {\it length} of $s$ as ${\rm lh}(s)=|\{t\in T:t<s\}|$. It is clear that $s<t$ implies ${\rm lh}(s)<{\rm lh}(t)$. For $n\in\omega$, we denote the $n$-th level of $T$ by
$$L_n(T)=\{s\in T:{\rm lh}(s)=n\}.$$
Each element in $L_0(T)$ is called a {\it root} of $T$.

Let $(T,<)$ be a tree. We say that $T$ is {\it well-founded} if any non-empty subset of $T$ contains at least a maximal element, or equivalently (under AC), if $T$ contains no infinite strictly increasing sequence. Let $T$ be a well-founded tree. We define the rank function $\rho_T$ on $T$ inductively as
$$\rho_T(s)=\sup\{\rho_T(t)+1:s<t\land t\in T\}$$
for $s\in T$. If $\rho_T(s)=0$, we say that $s$ is a {\it terminal} of $T$. Next we define
$$\rho(T)=\sup\{\rho_T(s)+1:s\in T\}.$$
So $\rho(T)=0$ iff $T=\emptyset$. It is clear that $\rho(T)=\sup\{\rho_T(s)+1:s\in L_0(T)\}$. If $L_0(T)=\{s_0\}$ is a singleton, then $\rho(T)=\rho_T(s_0)+1$ is a successor ordinal.

For $s\in T$, we define
$$T_s=\{t\in T:s=t\vee s<t\}.$$
Since $L_0(T_s)=\{s\}$, we have $\rho(T_s)=\rho_T(s)+1$. For the convenience of discussion, we let $T_s=\emptyset$ whenever $s\notin T$. Note that $\rho(T_s)$ is always a non-limit ordinal regardless of $s\in T$ or not.


We list some facts about computing the rank of nodes in a well-founded tree.


\begin{proposition}\label{tree}
Let $T$ be a well-founded tree, then for all $s\in T$,
\begin{align*}
    \rho_T(s)&=\sup\{\rho_T(t)+1:s<t\land t\in T\land\lh(t)=\lh(s)+1\}.
\end{align*}
\end{proposition}


\begin{proposition}\label{L_k(T)}
    Let $(T,<)$ be a well-founded tree, $k\in\omega$. For any $s\in L_k(T)$ and $m>k$, we have
        \begin{align*}
            \rho_T(s)\leq\sup\{\rho_T(t)+1:t\in L_m(T)\}+(m-k-1).
        \end{align*}
\end{proposition}


Let $(S,<)$ and $(T,<)$ be two trees. A map $\phi:S\to T$ is said to be an {\it order-preserving map} if
$$\forall s,t\in S\,(s<t\implies\phi(s)<\phi(t)).$$
It is said to be an {\it order-preserving embedding (isomorphism)} if it is injective (bijective) and
$$\forall s,t\in S\,(s<t\iff\phi(s)<\phi(t)).$$
If an order-preserving map $\phi$ additionally satisfies $\lh(\phi(s))=\lh(s)$ for each $s\in T$, then we say $\phi$ is {\it Lipschitz}. Note that if there is an order-preserving map $\phi:S\to T$ and $T$ is well-founded, then $(S,<)$ is also well-founded; moreover, we have $\rho(S)\leq\rho(T)$.



Next, we recall the definition of \hier\alpha\ groups. We begin by reviewing orbit trees. Given a set $X$, let $\mathcal{E}=(E_n)$ be a decreasing sequence of equivalence relations on $X$. We define the {\it orbit tree} of $\mathcal{E}$ on $X$ by
$$\otr{E}{X}=\{(n,C):n\in\omega\wedge C\mbox{ is a nonsingleton $E_n$-class}\},$$
ordered by
$$(n,C)<(m,D)\iff n<m\land C\supseteq D.$$

We recall a useful criterion to determine whether there exists an order-preserving map between orbit trees.

\begin{proposition}[{\cite[Proposition 3.4]{DingW}}]\label{uniform reduction induces tree embedding}
    Suppose $\mathcal{E}=(E_n)$ and $\mathcal{F}=(F_n)$ are two decreasing sequences of equivalence relations on $X$ and $Y$, respectively. Let $\theta:X\to Y$ be an injection. Then
    \begin{enumerate}
        \item if $\theta$ is a homomorphism from $E_n$ to $F_n$ for each $n\in\omega$, then there exists an order-preserving map from $\otr{E}{X}$ to $\otr{F}{Y}$;
        \item if $\theta$ is a reduction from $E_n$ to $F_n$ for each $n\in\omega$, then there exists a Lipschitz embedding from $\otr{E}{X}$ to $\otr{F}{Y}$; in particular, if $\theta$ is a bijection, then $\otr{E}{X}$ is isomorphic to $\otr{F}{Y}$.
    \end{enumerate}
\end{proposition}

For a non-archimedean Polish group $G$ and a countable discrete Polish $G$-space $X$, first we denote by $\dgnb(G)$ the set of all decreasing sequences $\mathcal{G}=(G_n)$ of open subgroups of $G$, such that $G_0=G$ and $(G_n)$ forms a neighborhood basis of $1_G$. Then by letting $\mathcal{E}=(E_n)=(E_{G_n}^X)$, we define the orbit tree of $\mathcal{G}$ on $X$ as
$$\otr{G}{X}=\otr{E}{X}.$$


We list some important results in \cite{DingW} concerning a non-archimedean Polish group $G$ and its orbit tree $\otr{G}{X}$.

\begin{theorem}[{\cite[Corollary 3.13]{DingW}}]\label{relationship between CLI groups and orbit trees}
    Let $G$ be a non-archimedean Polish group and $\mathcal{G}=(G_n)\in\dgnb(G)$. Then
    \begin{align*}
        &G\itx{is CLI}\\
        \iff&\otr{G}{X}\itx{is well-founded for any countable discrete Polish}G\itxl{-space}X\\
        \iff&\otr{G}{X(\mathcal{G})}\itx{is well-founded, where}X(\mathcal{G})=\bigcup_{n\in\omega}G/G_n.
    \end{align*}
\end{theorem}

\begin{theorem}[theorems 3.14 and 3.17 of \cite{DingW}]
    Let $G$ be a non-archimedean Polish group. For any $\mathcal{G}\in\dgnb(G)$, we define $\rho(\mathcal{G})=\rho(\otr{G}{X(\mathcal{G})})$. Then
    \begin{enumerate}
        \item $\omega(\rho(\mathcal{G}))$ is independent of the choice of $\mathcal{G}\in\dgnb(G)$; i.e., for any $\mathcal{G},\mathcal{G}'\in\dgnb(G)$, we have $\omega(\rho(\mathcal{G}))=\omega(\rho(\mathcal{G}'))$;
        \item whether $\rho(\mathcal{G})$ is a limit ordinal is independent of the choice of $\mathcal{G}\in\dgnb(G)$; i.e., for any $\mathcal{G},\mathcal{G}'\in\dgnb(G)$, we have
        $$\rho(\mathcal{G})\itxr{is a limit ordinal}\iff\rho(\mathcal{G}')\itxr{is a limit ordinal}.$$
    \end{enumerate}
\end{theorem}
In accordance with this theorem, for a non-archimedean CLI Polish group $G$, we can define
$$\rank(G)=\itxl{the unique ordinal}\beta\itx{such that}\omega(X(\mathcal{G}))=\omega\cdot\beta.$$
Furthermore, we can introduce the notion of \hier\alpha\space and L-\hier\alpha.
Let $\mathcal{G}\in\dgnb(G)$ and $\alpha<\omega_1$. We say $G$ is \textit{\hier\alpha} if $\rho(\mathcal{G})\leq\omega\cdot\alpha$; and we say $G$ is \textit{L-\hier\alpha} if $\omega(\rho(\mathcal G))\leq\omega\cdot\alpha$, i.e., $\rank(G)\leq\alpha$. Note that L-$\alpha$-CLI groups are clearly $(\alpha+1)$-CLI groups. If $G$ is \hier\alpha\space but not L-\hier\beta\space for any $\beta<\alpha$, we call $G$ {\it proper} \hier\alpha. Similarly, we call $G$ {\it proper} L-\hier\alpha\space if it is L-\hier\alpha\space but not \hier\alpha.

\begin{theorem}[theorems~3.19, 3.21, 4.1, and 4.2 of \cite{DingW}]
    Let $G$ be a non-archimedean CLI Polish group. Then
    \begin{enumerate}
        \item $G$ is \hier0\space iff $G=\{1_G\}$;
        \item $G$ is L-\hier0\space iff $G$ is countable discrete;
        \item $G$ is \hier1\space iff $G$ is TSI, i.e., $G$ admits a compatible two-sided invariant metric, hence all abelian non-archimedean Polish groups are \hier1;
        \item for any $\alpha<\omega_1$, $G$ is L-\hier\alpha\space iff $G$ has an open subgroup that is \hier\alpha;
        \item if $G$ is \hier\alpha\space for some $\alpha<\omega_1$, then any closed subgroup and any quotient group of $G$ are \hier\alpha.
    \end{enumerate}
\end{theorem}

Finally, we recall the definition of $\alpha$-balanced group and some of its basic properties (cf. \cite[Section 2]{SAllisonAPana}). Given a topological group $G$ and $V\subseteq G$, we use the notation $V\subseteq_1G$ to indicate that $V$ is an open neighborhood of $1_G$. We inductively define a relation $\rk(V,U)\leq\beta$ for $V,U\subseteq_1 G$ and an ordinal $\beta$ as follows:
\begin{enumerate}
    \item[(1)] $\rk(V,U)=0$ if $U\subseteq V$;
    \item[(2)] $\rk(V,U)\leq\beta$ for $\beta>0$ if there is a $W\subseteq_1G$ such that $$\forall\,g\in G\left(\rk(V,gWg^{-1})<\beta\right);$$
    \item[(3)] $\rk(V,U)=\infty$ if there is no ordinal $\beta$ with $\rk(V,U)\leq\beta$.
\end{enumerate}
We define $\rk(V,U)$ to be the least ordinal $\beta$ such that $\rk(V,U)\leq\beta$. Now let
\begin{equation*}
    \rk(G)=\sup\{\rk(V,G)+1:V\subseteq_1G\}.
\end{equation*}
For any ordinal $\alpha$, we say $G$ is $\alpha$-balanced iff $\rk(G)\leq\alpha$. 

The following proposition lists some basic properties of $\rk(V,U)$.

\begin{proposition}[{\cite[Proposition 2.3]{SAllisonAPana}}]\label{basic property of rk}
    Let $G$ be a topological group and let $U,V,U',V'\subseteq_1G$,  $h\in G$. Then
    \begin{enumerate}
        \item[(1)] if $V\supseteq V'$ and $U\subseteq U'$, then $\rk(V,U)\leq\rk(V',U')$;
        \item[(2)] $\rk(hVh^{-1},hUh^{-1})=\rk(V,U)$; and
        \item[(3)] $\rk(V\cap V',U)\leq\max\{\rk(V,U),\rk(V',U)\}$.
    \end{enumerate}
\end{proposition}


\section{$\alpha$-balanced and \hier\alpha}

In this section, we will discuss the relationship between $\alpha$-balanced and \hier\alpha. To do so, we need to give some lemmas first.

\begin{lemma}\label{subaction induced orbit tree}
    Suppose $G$ is a non-archimedean CLI Polish group, $H$ is a closed subgroup of $G$, and $X$ is a countable discrete Polish $G$-space.
    Let $\mathcal{G}=(G_n)\in\dgnb(G)$ and $\mathcal{H}=(H_n)=(G_n\cap H)\in\dgnb(H)$. Then $\rho(T_\mathcal{H}^X)\leq\rho(T_\mathcal{G}^X)$.
\end{lemma}

\begin{proof}
    Define $\theta:X\to X$ as $\theta(x)=x$. For each $n\in\omega$, let $E_n=E_{H_n}^X$ and $F_n=E_{G_n}^X$. Clearly, $\theta$ is a homomorphism from $E_n$ to $F_n$ for each $n\in\omega$. So applying Proposition~\ref{uniform reduction induces tree embedding}(1) we can see that there is an order-preserving map from $\otr{H}{X}$ to $\otr{G}{X}$. Hence $\rho(T_\mathcal{H}^X)\leq\rho(T_\mathcal{G}^X)$.
\end{proof}

For a closed subgroup $G$ of $S_\infty$, recall that
$$\gnb{n}=\{f\in G:\forall\,i<n\left(f(i)=i\right)\}.$$
For any $a\in\omega$, we denote by $G_a$ the stabilizer of $a$, i.e.,
$$G_a=\{g\in G:g(a)=a\}.$$
It follows that $\gnb{n}=G_0\cap\cdots\cap G_{n-1}$. To avoid confusion between the notation $G_a$ for stabilizer and the notation $G_n$ for members of a decreasing group neighborhood basis, we shall temporarily use only the former notation throughout the rest of this section.
\begin{lemma}\label{another highest orbit tree}
    Let $G$ be a closed subgroup of $S_\infty$, and $\mathcal{G}=(G_{\langle n\rangle})\in\dgnb(G)$. Then we have
    $$\rho(\mathcal{G})=\rho(\otr{G}{\omega}).$$
\end{lemma}
\begin{proof}
Fix a $k\geq 1$, and define $\delta:G\to G^k$ as $\delta(g)=(g,g,\dots,g)$. Clearly, $\delta(G)$ is a closed subgroup of $G^k$. Let $\mathcal{G}^k=(\gnb{n}^k)\in\dgnb(G^k)$ and $\mathcal{H}^k=(\delta(G)\cap\gnb{n}^k)\in\dgnb(\delta(G))$. By \cite[Lemma 4.3]{DingW}, we obtain $\rho(T_{\mathcal{G}^k}^{\omega^k})=\rho(\otr{G}{\omega})$. Then according to Lemma~\ref{subaction induced orbit tree}, we have $\rho(T_{\mathcal{H}^k}^{\omega^{k}})\leq\rho(\otr{G}{\omega})$.

Consider an action of $G$ on $\omega^k$, which is defined as
    $$g\cdot(a_0,\dots,a_{k-1})=(g(a_0),\dots,g(a_{k-1}))$$
    for $g\in G$ and $(a_0,\dots,a_{k-1})\in\omega^k$.
    Let $E_n=E_{G_{\langle n\rangle}}^{\omega^k}$ and $F_n=E_{\delta(G)\cap G_{\langle n\rangle}^k}^{\omega^k}$.
    It is easy to see that $E_n$ is isomorphic to $F_n$ for each $n\in\omega$, witnessed by the identity map on $\omega^k$. Therefore, using Proposition~\ref{uniform reduction induces tree embedding}(2), we can see that $\otr{G}{\omega^k}$ is isomorphic to $T_{\mathcal{H}^k}^{\omega^k}$.

Next, define $\theta:G/G_{\langle k\rangle}\to\omega^k$ as $\theta(gG_{\langle k\rangle})=(g(0),\dots,g(k-1))$. For each $n\in\omega$, let $E_n'=E_{G_{\langle n\rangle}}^{G/G_{\langle k\rangle}}$ and $F_n'=E_{G_{\langle n\rangle}}^{\omega^k}$. Then for any $n\in\omega$, it is straight forward to check that, $\theta$ is a reduction of $E_n'$ to $F_n'$. So by Proposition~\ref{uniform reduction induces tree embedding} again, we have
    $$\rho(T_{\mathcal{G}}^{G/G_{\langle k\rangle}})\leq\rho(T_{\mathcal{G}}^{\omega^k})
    =\rho(T_{\mathcal{H}^k}^{\omega^k}).$$

Consequently, we have $\rho(\mathcal{G})=\sup_k\rho(T_{\mathcal{G}}^{G/G_{\langle k\rangle}})\leq\rho(\otr{G}{\omega})$. Conversely, by \cite[Corollary 3.12]{DingW} we have $\rho(\otr{G}{\omega})\leq\rho(\mathcal{G})$. Thus $\rho(\mathcal{G})=\rho(\otr{G}{\omega})$.
\end{proof}

\begin{lemma}\label{computation of open set rank}
    Let $G$ be a closed subgroup of $S_\infty$. Then for any ordinal $\alpha>0$ and $a,n\in\omega$, we have
    \begin{align*}
        \rk(G_a,\gnb{n})\leq\alpha&\iff\exists\,m\geq n\,\forall\,b\in\gnb{n}\cdot a\,(\rk(G_b,\gnb{m})<\alpha).
    \end{align*}
\end{lemma}
\begin{proof}
    $(\Rightarrow)$. Since $\rk(G_a,\gnb{n})\leq\alpha$, there exists a $W\subseteq_1G$ such that for all $g\in\gnb{n}$, $\rk(G_a,gWg^{-1})<\alpha$. Now we can find an $m\geq n$ with $\gnb{m}\subseteq W$. Then by Proposition~\ref{basic property of rk} we obtain
    \begin{align*}
        \rk(G_{g^{-1}(a)},\gnb{m})&=\rk(g^{-1}G_ag,\gnb{m})=\rk(G_a,g\gnb{m}g^{-1})\\
        &\leq\rk(G_a,gWg^{-1})<\alpha.
    \end{align*}
    So for any $b\in\gnb{n}\cdot a$ we have $\rk(G_b,\gnb{m})<\alpha$.

    $(\Leftarrow)$. For any $g\in\gnb{n}$, we have $\rk(G_a,g\gnb{m}g^{-1})=\rk(G_{g^{-1}(a)},\gnb{m})<\alpha$ by Proposition~\ref{basic property of rk}. Hence $\rk(G_a,\gnb{n})\leq\alpha$.
\end{proof}

\begin{lemma}\label{computation of group rank}
    Let $G$ be a closed subgroup of $S_\infty$. Then
    $$\rk(G)=\sup\{\rk(G_a,G)+1:a\in\omega\}.$$
\end{lemma}
\begin{proof}
    For any $V\subseteq_1G$, we can find an $m>0$ such that $\gnb{m}\subseteq V$. Now by Proposition~\ref{basic property of rk} we have
    $$\rk(V,G)\leq\rk(\gnb{m},G)\leq\max\{\rk(G_0,G),\dots,\rk(G_{m-1},G)\}.$$
    Thus
    \begin{align*}
        \rk(G)&=\sup\{\rk(V,G)+1:V\subseteq_1G\}\\
        &\leq\sup\{\rk(G_a,G)+1:a\in\omega\}\leq\rk(G).
    \end{align*}
    This shows $\rk(G)=\sup\{\rk(G_a,G)+1:a\in\omega\}$.
\end{proof}

We are now ready to describe the connection between $\alpha$-balanced groups and \hier\alpha\space groups.

\begin{theorem}\label{rank and balanced, finite case}
    Suppose $G$ is a closed subgroup of $S_\infty$. Let $\mathcal{G}=(\gnb{n})\in\dgnb(G)$ and $T=\otr{G}{\omega}$, then for any natural number $l\geq 1$, we have
    $$\rho_T(n,\gnb{n}\cdot a)<\omega\cdot l\iff\rk(G_a,\gnb{n})\leq l$$
    for $(n,\gnb{n}\cdot a)\in T$.
\end{theorem}
\begin{proof}
    We will prove the statement by induction on $l$.

    (1) $l=1$.

    ($\Rightarrow$). Suppose $\rho_T(n,\gnb{n}\cdot a)=k<\omega$, then for any $b\in\gnb{n}\cdot a$, we can see that $\gnb{n+k+1}\cdot b$ is a singleton. Hence $G_b\supseteq\gnb{n+k+1}$. This means $\rk(G_b,\gnb{n+k+1})=0$. It follows from Lemma~\ref{computation of open set rank} that $\rk(G_a,\gnb{n})\leq1$.

    ($\Leftarrow$). Since $\rk(G_a,\gnb{n})\leq 1$, by Lemma~\ref{computation of open set rank} we can find an $m\geq n$, such that for any $b\in\gnb{n}\cdot a$, $G_b\supseteq\gnb{m}$ holds. Hence for any $b\in\gnb{n}\cdot a$, we have $\gnb{m}\cdot b=\{b\}$. Applying Proposition~\ref{L_k(T)}, we can deduce that
    $$\rho_T(n,\gnb{n}\cdot a)\leq\max\{0,m-n-1\}<\omega.$$

    (2) $l>1$.

    ($\Rightarrow$). Suppose $\rho_T(n,\gnb{n}\cdot a)=\omega\cdot(l-1)+k$, then for any $m>n+k$ and any $b\in\gnb{n}\cdot a$, $\rho_T(m,\gnb{m}\cdot b)<\omega\cdot(l-1)$ holds. From the induction hypothesis we obtain $\forall\,b\in\gnb{n}\cdot a\,(\rk(G_b,\gnb{m})\leq l-1)$. Thus by Lemma~\ref{computation of open set rank}, $\rk(G_a,\gnb{n})\leq l$ holds.

    ($\Leftarrow$). Since $\rk(G_a,\gnb{n})\leq l$, by definition and Proposition~\ref{basic property of rk}(1), we can find an $m>n$ such that for any $g\in\gnb{n}$, $\rk(G_a,g\gnb{m}g^{-1})\leq l-1$ holds. Next, using Proposition~\ref{basic property of rk}(2) we obtain $\rk(G_{g^{-1}(a)},\gnb{m})\le l-1$. This implies $\forall\,b\in\gnb{n}\cdot a\,(\rk(G_b,\gnb{m})\le l-1)$. According to the induction hypothesis, we have $\rho_T(m,\gnb{m}\cdot b)<\omega\cdot(l-1)$ for $b\in\gnb{n}\cdot a$. Then using Proposition~\ref{L_k(T)},
    \begin{align*}
        \rho_T(n,\gnb{n}\cdot a)&\leq\sup\{\rho_T(m,\gnb{m}\cdot b)+1:b\in\gnb{n}\cdot a\}+(m-n-1)\\
        &\leq\omega\cdot(l-1)+(m-n-1)<\omega\cdot l.
    \end{align*}
   This finishes the induction.
\end{proof}

\begin{theorem}\label{rank and balanced, infinite case}
    Suppose $G$ is a closed subgroup of $S_\infty$. Let $\mathcal{G}=(\gnb{n})\in\dgnb(G)$ and $T=\otr{G}{\omega}$, then for any ordinal $\alpha\geq\omega$, we have
    $$\rho_T(n,\gnb{n}\cdot a)<\omega\cdot(\alpha+1)\iff\rk(G_a,\gnb{n})\leq\alpha.$$
\end{theorem}
\begin{proof}
We will prove the statement by induction on $\alpha\geq\omega$.

(1) $\alpha=\omega$.

$(\Rightarrow)$. Without loss of generality, we can assume that $\rho_T(n,\gnb{n}\cdot a)=\omega\cdot\omega+k$ for some $k<\omega$. Then for any $m>n+k$ and any $b\in\gnb{n}\cdot a$, $\rho_T(m,\gnb{m}\cdot b)<\omega\cdot l_b$ for some $0<l_b<\omega$. By Theorem~\ref{rank and balanced, finite case}, we have $\rk(G_b,\gnb{m})\leq l_b$ for all $b\in\gnb{n}\cdot a$. Next applying Lemma~\ref{computation of open set rank}, we can see that $\rk(G_a,\gnb{n})\leq\omega$.

$(\Leftarrow)$. By Lemma~\ref{computation of open set rank} and Proposition~\ref{basic property of rk}(1), there exists an $m>n$ such that for all $b\in\gnb{n}\cdot a$, $\rk(G_b,\gnb{m})\leq l_b$ for some $0<l_b<\omega$. So by Theorem~\ref{rank and balanced, finite case} we have $\rho_T(m,\gnb{m}\cdot b)<\omega\cdot l_b$. Applying Proposition~\ref{L_k(T)}, we can obtain that
    \begin{align*}
        &\rho_T(n,\gnb{n}\cdot a)\\
        \leq\,&\sup\{\rho_T(m,\gnb{m}\cdot b)+1:b\in\gnb{n}\cdot a\}+(m-n-1)\\
        \leq\,&\omega\cdot\omega+(m-n-1)\\
        <\,&\omega\cdot(\omega+1).
    \end{align*}

    (2) $\alpha>\omega$ is a successor ordinal.

    $(\Rightarrow)$. Let $\alpha=\beta+1$. Note that there is an $m>n$ such that for all $b\in\gnb{n}\cdot a$, $\rho_T(m,\gnb{m}\cdot b)<\omega\cdot\alpha=\omega\cdot(\beta+1)$ holds. Then applying the induction hypothesis we can see that $\rk(G_b,\gnb{m})\leq\beta$ for each $b\in\gnb{n}\cdot a$. Thus $\rk(G_a,\gnb{n})\leq\alpha$ follows from Lemma~\ref{computation of open set rank}.

    $(\Leftarrow)$. By Lemma~\ref{computation of open set rank} and Proposition~\ref{basic property of rk}(1) again, we can choose an $m>n$ so that
    $$\forall\,b\in\gnb{n}\cdot a\,(\rk(G_b,\gnb{m})\leq\beta).$$
    Now by induction hypothesis we have $\rho_T(m,\gnb{m}\cdot b)<\omega\cdot(\beta+1)$. Hence using Proposition~\ref{L_k(T)} we obtain
    \begin{align*}
        \rho_T(n,\gnb{n}\cdot a)&\leq\sup\{\rho_T(m,\gnb{m}\cdot b)+1:b\in\gnb{n}\cdot a\}+(m-n-1)\\
        &\leq\omega\cdot(\beta+1)+(m-n-1)\\
        &<\omega\cdot(\alpha+1).
    \end{align*}

    (3) $\alpha>\omega$ is a limit ordinal.

    $(\Rightarrow)$. Choose an $m>n$ such that for any $b\in\gnb{n}\cdot a$, $\rho_T(m,\gnb{m}\cdot b)<\omega\cdot\alpha$. Then there is an ordinal $\omega\leq\beta_b<\alpha$ such that $\rho_T(m,\gnb{m}\cdot b)<\omega\cdot(\beta_b+1)$. From the induction hypothesis we obtain $\rk(G_b,\gnb{m})\leq\beta_b<\alpha$. Thus $\rk(G_a,\gnb{n})\leq\alpha$.

    $(\Leftarrow)$. Similarly, we can find an $m>n$ such that for any $b\in\gnb{n}\cdot a$, $\rk(G_b,\gnb{m})<\alpha$. Let $\beta_b<\alpha$ be an ordinal with $\rk(G_b,\gnb{m})\leq\beta_b$. Then from the induction hpothesis we can see that $\rho_T(m,\gnb{m}\cdot b)<\omega\cdot(\beta_b+1)$. Therefore,
    \begin{align*}
        \rho_T(n,\gnb{n}\cdot a)&\leq\sup\{\rho_T(m,\gnb{m}\cdot b)+1:b\in\gnb{n}\cdot a\}+(m-n-1)\\
        &\leq\sup\{\omega\cdot(\beta_b+1):b\in\gnb{n}\cdot a\}+(m-n-1)\\
        &\leq\omega\cdot\alpha+(m-n-1)<\omega\cdot(\alpha+1).
    \end{align*}
\end{proof}

\begin{corollary}
    Let $G$ be a closed subgroup of $S_\infty$,  $\omega\leq\alpha<\omega_1$ be an ordinal, $n\in\omega$. Then
    \begin{enumerate}
        \item[(1)] $G$ is \hier{n}\space iff $G$ is $(n+1)$-balanced;
        \item[(2)] $G$ is \hier\alpha\space iff $G$ is $\alpha$-balanced.
    \end{enumerate}
\end{corollary}
\begin{proof}
Let $\mathcal{G}=(\gnb{n})\in\dgnb(G)$ and $T=\otr{G}{\omega}$.

(1)$(\Rightarrow)$. By definition and Lemma~\ref{another highest orbit tree} we have $\rho(T)=\rho(\mathcal{G})\leq\omega\cdot n$. So we have $\rho_T(0,G\cdot a)<\omega\cdot n$ for each $a\in\omega$. Then from Theorem~\ref{rank and balanced, finite case} we can see that $\rk(G_a,G)\leq n$ for any $a\in\omega$. It follows from Lemma~\ref{computation of group rank} that $G$ is $(n+1)$-balanced.

$(\Leftarrow)$. For any $a\in\omega$, we have $\rk(G_a,G)\leq n$. So by Theorem~\ref{rank and balanced, finite case}, $\rho_T(0,G\cdot a)<\omega\cdot n$ holds. Hence $\rho(T)\leq\omega\cdot n$. This means $G$ is \hier n.

(2) Case 1. $\alpha=\omega$.

$(\Rightarrow)$. By Lemma~\ref{another highest orbit tree}, for any $a\in\omega$, we have $\rho_T(0,G\cdot a)<\omega\cdot\omega$. So we can find an $l_a\in\omega$ such that $\rho_T(0,G\cdot a)<\omega\cdot l_a$. Then Theorem~\ref{rank and balanced, finite case} implies that $\rk(G_a,G)\leq l_a$. Hence $G$ is $\omega$-balanced by Lemma~\ref{computation of group rank}.

$(\Leftarrow)$. For any $a\in\omega$, we have $\rk(G_a,G)=l_a<\omega$ for some $l_a$. So we can get $\rho_T(0,G\cdot a)\leq\omega\cdot l_a$ from Theorem~\ref{rank and balanced, finite case}. This shows that $\rho(T)\leq\omega\cdot\omega$, i.e., $G$ is \hier\omega.

Case 2. $\alpha>\omega$.

$(\Rightarrow)$. By definition and Lemma~\ref{another highest orbit tree} again we obtain $\rho(T)\leq\omega\cdot\alpha$. So for any $a\in\omega$, we have $\rho_T(0,G\cdot a)<\omega\cdot\alpha$. Let $\rho_T(0,G\cdot a)=\omega\cdot\beta_a+k_a$ with $k_a<\omega$. Note that $\beta_a<\alpha$. Now for those $a\in\omega$ with $\beta_a\geq\omega$, we have $\rk(G_a,G)\leq\beta_a$ by Theorem~\ref{rank and balanced, infinite case}. Therefore,
\begin{align*}
    \rk(G)&=\sup\{\rk(G_a,G)+1:a\in\omega\}\\
    &\leq\sup\{\omega,\beta_a+1:a\in\omega\land\beta_a\geq\omega\}\\
    &\leq\alpha,
\end{align*}
i.e., $G$ is $\alpha$-balanced.

$(\Leftarrow)$. By Case 1, we can assume that $\rk(G)>\omega$. For each $a\in\omega$, let $\rk(G_a,G)=\beta_a$. Then from Lemma~\ref{computation of group rank}, we have
$$\omega<\sup\{\beta_a+1:a\in\omega\}=\rk(G)\leq\alpha;$$
and for those $a\in\omega$ with $\beta_a\geq\omega$, we can obtain $\rho_T(0,G\cdot a)<\omega\cdot(\beta_a+1)$ by Theorem~\ref{rank and balanced, infinite case}. So
\begin{align*}
        \rho(T)&=\sup\{\rho_T(0,G\cdot a)+1:a\in\omega\}\\
        &\leq\sup\{\omega\cdot(\beta_a+1):a\in\omega\}\\
        &\leq\omega\cdot\alpha.
\end{align*}
This shows that $G$ is \hier\alpha.
\end{proof}

We can combine these results into a single statement.
\begin{corollary}\label{relation of balanced and CLI, closed subgroup of Sinfty}
    Let $G$ be a closed subgroup of $S_\infty$, $\alpha<\omega_1$. Then $G$ is \hier\alpha\space iff $G$ is $(1+\alpha)$-balanced.
\end{corollary}

Since each non-archimedean Polish group $G$ embeds as a closed subgroup of $S_\infty$, we conclude the following theorem.
\begin{theorem}\label{relation of balanced and CLI, non-archimedean case}
    Let $G$ be a non-archimedean Polish group, $\alpha<\omega_1$. Then $G$ is \hier\alpha\space iff $G$ is $(1+\alpha)$-balanced.
\end{theorem}

We now turn to a problem that was asked by Allison and Panagiotopoulos in \cite{SAllisonAPana}.
\begin{problem}\label{problem of SP}
    Let $H$ be a closed normal subgroup of a topological group $G$. Is it true that
    \begin{equation*}
        \rk(G)\leq\sup\{\rk(V,H;H)+\rk(G/H):V\subseteq_1H\}?
    \end{equation*}
Here $\rk(V,H;H)$ means that the underlying topological group for function $\rk(V,H)$ is $H$.
\end{problem}
Inspired by this problem, we consider an analogous problem under the setting of \hier\alpha.
\begin{problem}\label{problem of DW}
    Let $N$ be a closed normal subgroup of a non-archimedean Polish group $G$. Suppose $N$ is \hier\alpha\space and $G/N$ is \hier\beta, $\alpha,\beta<\omega_1$. Is $G$ \hier{(\alpha+\beta)}?
\end{problem}
Assume that $N,G/N,G$ are proper \hier\alpha, \hier\beta, \hier\gamma, respectively. If Problem~\ref{problem of SP} has a positive answer, then applying Theorem~\ref{relation of balanced and CLI, non-archimedean case} yields
\begin{align*}
    1+\gamma&=\rk(G)\leq\sup\{\rk(V,H;H)+\rk(G/H):V\subseteq_1H\}\\
    &=\sup\{\rk(V,H;H)+1+\beta:V\subseteq_1H\}\\
    &\leq\rk(H)+\beta=1+\alpha+\beta.
\end{align*}
So $\gamma\leq\alpha+\beta$. This observation shows that if Problem~\ref{problem of SP} has a positive answer, then so does Problem~\ref{problem of DW}.

In the following section we devote ourselves to solve Problem~\ref{problem of DW}.

\section{Positive Results}\label{Examples and counterexamples}

In this section, we provide a positive answer to Problem~\ref{problem of DW} under an additional assumption.

Recall the definition of semidirect product of groups. For two non-archimedean Polish groups $G$, $H$ and a homomorphism $\rho:G\to\aut(H)$, the {\it semidirect product} $G\ltimes_\rho H$ is the set $G\times H$ together with the multiplication
$$(x_1,y_1)(x_2,y_2)=(x_1x_2,y_1\rho(x_1)(y_2)),$$
where $(x_1,y_1),(x_2,y_2)\in G\times H$. In addition, if the map $(x,y)\mapsto\rho(x)(y)$ is continuous, then the following proposition shows that, $G\ltimes_\rho H$ equipped with the product topology is also a non-archimedean Polish group.
\begin{proposition}\label{dgnb form of semidirect product}
    Let $G,H$ be two non-archimedean Polish groups, and $\rho:G\to\mathrm{Aut}(H)$ be a homomorphism. Let $A=G\ltimes_\rho H$ and endow it with the product topology. Suppose the map $(x,y)\mapsto\rho(x)(y)$ is continuous, then $(1_G,1_H)$ has a neighborhood basis consisting of open subgroups of the form $G'\times H'$, where $G'$ and $H'$ are open subgroups of $G$ and $H$ respectively.
\end{proposition}
\begin{proof}
    Let $(G_n)\in\dgnb(G)$, $(H_n)\in\dgnb(H)$. Fix an $n\in\omega$, and define $H_n'=\bigcup_{k\in\omega}(\rho(G_n)(H_n))^k$. We claim that $G_n\times H_n'$ is a subgroup of $G\times H$.

    For $(g,a),(h,b)\in G_n\times H_n'$, there are $k,l\in\omega$ such that
    \begin{align*}
        a&=\rho(x_1)(y_1)\cdots\rho(x_k)(y_k), \\
        b&=\rho(x_1')(y_1')\cdots\rho(x_l')(y_l')
    \end{align*}
    for some $x_1,\dots,x_k,x_1',\dots,x_l'\in G_n$ and $y_1,\dots,y_k,y_1',\dots,y_l'\in H_n$. Thus we have
    \begin{align*}
        (g,a)(h,b)&=(gh,a\rho(g)(\rho(x_1')(y_1')\cdots\rho(x_l')(y_l'))) \\
                  &=(gh,a\rho(gx_1')(y_1')\cdots\rho(gx_l')(y_l')) \\
                  &\in G_n\times(\rho(G_n)(H_n))^{k+l},
    \end{align*}
    and
    \begin{align*}
        (g,a)^{-1}&=(g^{-1},\rho(g^{-1})(\rho(x_k)(y_k)^{-1}\cdots\rho(x_1)(y_1)^{-1})) \\
                  &=(g^{-1},\rho(g^{-1}x_k)(y_k^{-1})\cdots\rho(g^{-1}x_1)(y_1^{-1})) \\
                  &\in G_n\times(\rho(G_n)(H_n))^k.
    \end{align*}
    This proves the claim.

Now since $(x,y)\mapsto\rho(x)(y)$ is continuous, for each $n\in\omega$ there is a $k\ge n$ such that $\rho(G_k)(H_k)\subseteq H_n$. This implies that $G_k\times H_k'\subseteq G_n\times H_n$. Also, note that $G_k\times H_k'$ is open, so we can find an $l\in\omega$ such that $G_l\times H_l\subseteq G_k\times H_k'$.
Using these facts interchangeably, we can find a strictly increasing sequence $(k_n)$ such that $(G_{k_n}\times H_{k_n}')\in\mathrm{dgnb}(A)$.
So $(G_{k_n}\times H_{k_n}')$ is the desired neighborhood basis of $(1_G,1_H)$.
\end{proof}


Next we give a useful lemma in estimating the rank of the semidirect product of non-archimedean Polish groups.
\begin{lemma}\label{rank of concatenated trees}
    Let $(T,<)$ be a well-founded tree, $\alpha<\omega_1$ and $T_1\subseteq T$. Suppose $T_1$ is downward closed under $<$, i.e.,
    $$\forall\,x\in T_1\,\forall\,y\in T\left(y<x\implies y\in T_1\right).$$
    Assume additionally that for any $y\in T\setminus T_1$ we have $\rho_T(y)<\alpha$. Then for $x\in T_1$, we have
    $$\rho_T(x)\leq\alpha+\rho_{T_1}(x).$$
    In particular, $\rho(T)\leq\alpha+\rho(T_1)$.
\end{lemma}
\begin{proof}
    The proof goes by induction on $\rho_{T_1}(x)$ for $x\in T_1$. First, if $\rho_{T_1}(x)=0$, then either $\rho_T(x)=0$, which obviously satisfies the inequality; or $\rho_T(x)>0$, then
    $$\rho_T(x)=\sup\{\rho_T(y)+1:y>x\land y\in T\}.$$
    Since $\rho_{T_1}(x)=0$, we can see that $y>x\land y\in T$ implies $y\in T\setminus T_1$. So $\rho_T(y)+1\leq\alpha$. This shows that $\rho_T(x)\leq\alpha=\alpha+\rho_{T_1}(x)$. The indution step of the proof follows easily.

    Next we estimate the rank of the tree $T$. Clearly, we have
    \begin{align*}
        \rho(T)&=\sup\{\rho_T(x)+1:x\in T\}\\
        &=\sup\left(\{\rho_T(x)+1:x\in T_1\}\cup\{\rho_T(x)+1:x\in T\setminus T_1\}\right)\\
        &\leq\sup\{\alpha,\alpha+\rho_{T_1}(x)+1:x\in T_1\}\\
        &=\alpha+\rho(T_1).
    \end{align*}
\end{proof}

\begin{theorem}\label{positive answer for semidirect product}
    Let $G,H$ be two non-archimedean Polish groups, and let $A=G\ltimes_\rho H$. Suppose that $G$ and $H$ are \hier\alpha\ and \hier\beta, respectively, for some $\alpha,\beta<\omega_1$. Choose $\mathcal{G}=(G_n)\in\dgnb(G)$ and $\mathcal{H}=(H_n)\in\dgnb(H)$ such that $\mathcal{A}=(A_n)=(G_n\times H_n)\in\dgnb(A)$. If, in addition, we have
    \begin{equation}
    \begin{aligned}
        \forall\,n\,\forall\,x\in G_n\,\forall\,y\in H\left(\rho(x)(yH_n)=yH_n\right).
    \end{aligned}
    \label{addition condition}
    \end{equation}
    Then $A$ is \hier{(\beta+\alpha)}.
\end{theorem}
\begin{proof}
    Fix an $n\in\omega$. Since $G_n\times H_n$ is a subgroup of $A$, we can see that $\rho(G_n)(H_n)=H_n$. For $(x_1,y_1)A_n,(x_2,y_2)A_n\in A/A_n$, we have
    \begin{equation}
        \begin{aligned}
            &(x_1,y_1)A_n=(x_2,y_2)A_n\\
            \iff&(x_2,y_2)^{-1}(x_1,y_1)=\left(x_2^{-1}x_1,\rho(x_2^{-1})(y_2^{-1}y_1)\right)\in A_n\\
            \iff&x_2^{-1}x_1\in G_n\land\rho(x_2^{-1})(y_2^{-1}y_1)\in H_n.
        \end{aligned}
        \label{coset equality}
    \end{equation}
    So for any $(a,b)\in A$ we have
    \begin{equation}
        \begin{aligned}
            &(a,b)(x_1,y_1)A_n=(x_2,y_2)A_n\\
            \iff&(ax_1,b\rho(a)(y_1))A_n=(x_2,y_2)A_n\\
            \iff&x_2^{-1}ax_1\in G_n\land\rho(x_2^{-1})(y_2^{-1}b\rho(a)(y_1))\in H_n.
        \end{aligned}
        \label{action of A equality}
    \end{equation}

We define
    $$T_1=\{(p,A_p\cdot (x,y)A_n)\in T_{\mathcal{A}}^{A/A_n}:G_p\cdot xG_n\itxr{is not a singleton}\}.$$
To see that $T_1$ is well defined, let $A_p\cdot (x_1,y_1)A_n=A_p\cdot (x_2,y_2)A_n$. There exists $(a,b)\in A_p$ with $(a,b)(x_1,y_1)A_n=(x_2,y_2)A_n$.
By (\ref{action of A equality}), we get $ax_1G_n=x_2G_n$, and hence $G_p\cdot x_1G_n=G_p\cdot x_2G_n$. This implies that $T_1$ is well defined.

    Now, for any $(p,A_p\cdot(x,y)A_n)\in T\setminus T_1$, from the definition of $T_1$ we can see that $G_{p}\cdot xG_n=\{xG_n\}.$
    So for any $a\in G_{p}$, we have $axG_n=xG_n$. Next using (\ref{action of A equality}), we can see that for $(x_1,y_1)A_n,(x_2,y_2)A_n\in A_{p}\cdot(x,y)A_n$ and $m\in\omega$,
    \begin{align*}
        &(x_1,y_1)A_nE_{A_{m+p}}^{A_{p}\cdot(x,y)A_n}(x_2,y_2)A_n\\
        \iff&\exists\,(a,b)\in A_{m+p}\left((a,b)(x_1,y_1)A_n=(x_2,y_2)A_n\right)\\
        \iff&\exists\,(a,b)\in A_{m+p}\left(ax_1G_n=x_2G_n\land y_2^{-1}b\rho(a)(y_1)\in\rho(x_2)(H_n)\right).
    \end{align*}
    Note that $x_2G_n\in G_{p}\cdot xG_n=\{xG_n\}$, so $x_2G_n=xG_n$. Thus
    \begin{align*}
        \rho(x_2)(H_n)&=\rho(x_2)(\rho(G_n)(H_n))=\rho(x_2G_n)(H_n)\\
        &=\rho(xG_n)(H_n)=\rho(x)(H_n).
    \end{align*}
    Let $H_n'=\rho(x)(H_n)$. Then for any $a\in G_{p}$, noting that $axG_n=xG_n$, so $x^{-1}ax\in G_n$. Thus by (\ref{addition condition}) we have
    \begin{align*}
        \rho(a)(y_1)H_n'&=\rho(a)(y_1)\rho(x)(H_n)=\rho(x)\left(\rho(x^{-1}a)(y_1)H_n\right)\\
        &=\rho(x)\left(\rho(x^{-1}a)(y_1)\rho(x^{-1}ax)(H_n)\right)\\
        &=\rho(x)\circ\rho(x^{-1}ax)\left(\rho(x^{-1})(y_1)H_n\right)\\
        &=\rho(x)\left(\rho(x^{-1})(y_1)H_n\right)=y_1H_n'.
    \end{align*}
    Hence
    \begin{align*}
        &(x_1,y_1)A_nE_{A_{m+p}}^{A_{p}\cdot(x,y)A_n}(x_2,y_2)A_n\\
        \iff&\exists\,(a,b)\in A_{m+p}\left(ax_1G_n=x_2G_n\land y_2H_n'=by_1H_n'\right)\\
        \iff&\exists\,b\in H_{m+p}\left(y_2H_n'=by_1H_n'\right)
    \end{align*}
    holds for any $m\in\omega$. This shows $\theta:A_{p}\cdot(x,y)A_n\to H_{p}\cdot yH_n'$ defined as
    $$\theta((x_1,y_1)A_n)=y_1H_n'$$
    is a reduction from $E_{A_{m+p}}^{A_{p}\cdot(x,y)A_n}$ to $E_{H_{m+p}}^{H_{p}\cdot yH_n'}$ for each $m\in\omega$.
    Now let $\mathcal{A}'=(A_{m+p})_{m\in\omega}$, $\mathcal{H}'=(H_{m+p})_{m\in\omega}$.
    By Proposition~\ref{uniform reduction induces tree embedding}, we can see that $\otr{A'}{A_{p}\cdot(x,y)A_n}$ is isomorphic to $T_{\mathcal{H'}}^{H_{p}\cdot yH_n'}$.
    Since $H_{p}$ is a clopen subgroup of $H$, we can see that $H_{p}$ is also \hier\beta. Then according to \cite[Lemma 3.11]{DingW}, there is a $k\in\omega$ such that $\rho(\otr{H'}{H_{p}\cdot yH_n'})\leq\rho(\otr{H'}{H_{p}/H_{k+p}})\leq\omega\cdot\beta$. But $\rho(\otr{H'}{H_{p}/H_{k+p}})$ is a successor ordinal, so $\rho(\otr{H'}{H_{p}/H_{k+p}})<\omega\cdot\beta$.
    This implies that $\rho(\otr{H'}{H_{p}\cdot yH_n'})<\omega\cdot\beta$. Therefore,
    $$\rho_T(p,A_p\cdot(x,y)A_n)=\rho(\otr{A'}{A_p\cdot(x,y)A_n})-1=\rho(\otr{H'}{H_p\cdot yH_n'})-1<\omega\cdot\beta.$$

    Next, let $S=\otr{G}{G/G_n}$. We will show the following equality
    $$\rho_{T_1}(p,A_p\cdot(x,y)A_n)=\rho_S(p,G_p\cdot xG_n)$$
    by induction on $\rho_{T_1}(p,A_p\cdot(x,y)A_n)$ for $(p,A_p\cdot(x,y)A_n)\in T_1$.

    If $\rho_{T_1}(p,A_p\cdot(x,y)A_n)=0$, then $(p,G_p\cdot xG_n)\in S$. Moreover, for any $x'G_n\in G_p\cdot xG_n$, we have $(p+1,A_{p+1}\cdot(x',y')A_n)\notin T_1$ for some $y'\in H$ with $(x',y')A_n\in A_p\cdot(x,y)A_n$, thus $G_{p+1}\cdot x'G_n$ is a singleton. This shows $\rho_S(p,G_p\cdot xG_n)=0$.

    Now for $(p,A_p\cdot(x,y)A_n)\in T_1$, suppose we have proved the equality for all $(q,O')>(p,A_p\cdot(x,y)A_n)$ in $T_1$. Then we have
    \begin{align*}
        &\rho_{T_1}(p,A_p\cdot(x,y)A_n)\\
        =&\sup\{\rho_{T_1}(p+1,A_{p+1}\cdot(x',y')A_n)+1:(x',y')A_n\in A_p\cdot(x,y)A_n\}\\
        =&\sup\{\rho_S(p+1,G_{p+1}\cdot x'G_n)+1:(x',y')A_n\in A_p\cdot(x,y)A_n\}\\
        =&\sup\{\rho_S(p+1,G_{p+1}\cdot x'G_n)+1:x'G_n\in G_p\cdot xG_n\}\\
        =&\rho_S(p,G_p\cdot xG_n).
    \end{align*}
    This finishes the induction.

    From the equality above we can see $\rho(T_1)=\rho(S)\leq\omega\cdot\alpha$. Then Lemma~\ref{rank of concatenated trees} gives $\rho(\otr{A}{A/A_n})\leq\omega\cdot\beta+\rho(\otr{G}{G/G_n})\leq\omega\cdot(\beta+\alpha)$.

    Finally, since $n$ is arbitrary, we have $\rho(\mathcal{A})\leq\omega\cdot(\beta+\alpha)$. This implies that $A$ is $(\beta+\alpha)$-CLI.
\end{proof}

We will give two applications of Theorem~\ref{positive answer for semidirect product} below. The first is the wreath products of non-archimedean Polish groups.

Let $G,H$ be two closed subgroups of $S_\infty$. Recall that the {\it wreath product} of $H$ by $G$ on $\omega$ (in symbol $G\wr H$) is the semidirect product $A=G\ltimes_\rho H^\omega$, where the homomorphism $\rho:G\to\aut(H^\omega)$ is given by
    $$\rho(g)(F)=F\circ g^{-1}$$
    for $(g,F)\in A$. Clearly, $A$ is a non-archimedean Polish group. Take $\mathcal{G}=(\gnb{n})\in\dgnb(G)$ and $\mathcal{H}=(H_n)\in\dgnb(H)$. Let
    $$\mathcal{A}=(A_n)=\left(\gnb{n}\times\prod_{i<n}H_n\times\prod_{i\geq n}H\right)\in\dgnb(A).$$
    We check that $G\wr H$ satisfies the condition~(\ref{addition condition}):

    Fix an $n\in\omega$, and let $K_n=\prod_{i<n}H_n\times\prod_{i\geq n}H$. Take any $g\in\gnb{n}$ and $F\in H^\omega$. Note that $\rho(g)(K_n)=K_n$. Then we have
    \begin{align*}
        \rho(g)(FK_n)=\rho(g)(F)\rho(g)(K_n)=\rho(g)(F)K_n=FK_n.
    \end{align*}
This gives the following result.

\begin{theorem}
    Let $G$ and $H$ be two closed subgroups of $S_\infty$. Suppose $G$ and $H$ are \hier\alpha\ and \hier\beta, respectively. Then $G\wr H$ is \hier{(\beta+\alpha)}.
\end{theorem}

\begin{proof}
From~\cite[Theorem 4.6(1)]{DingW}, $H^\omega$ is also \hier\beta.
\end{proof}

Next, we consider the semidirect product of $\Z^\omega$. Recall that
$$\aut(\Z^\omega)=\{f:f,f^{-1}\itx{are both continuous automorphisms on}\Z^\omega\}.$$
We will show that any closed subgroup $G$ of $\aut(\Z^\omega)$ satisfies the condition in Theorem~\ref{positive answer for semidirect product}. Put $L=\aut(\Z^\omega)$. For each $n\in\omega$, define $e_n\in\Z^\omega$ as
\begin{equation*}
    e_n(i)=\begin{cases}
        0,&\itx{if}i\neq n,\\
        1,&\itx{if}i=n.
    \end{cases}
\end{equation*}
Then we can treat any element of $\Z^\omega$ as a column vector and any $f\in\aut(\Z^\omega)$ as an infinite matrix $(f(e_0),f(e_1),\dots)$ over $\Z$. Clearly, each infinite matrix $f=(a_{ij})\in\aut(\Z^\omega)$ is row-finite, i.e.,
$$\forall\,i\in\omega\,\exists\,N\geq 1\,\forall\,j\geq N(a_{ij}=0).$$

Now, endow $\aut(\Z^\omega)$ with the compact-open topology $\mathcal{T}$, and let
$$C(K,V)=\{f\in\aut(\Z^\omega):f(K)\subseteq V\}$$
for compact $K\subseteq\Z^\omega$ and open $V\subseteq\Z^\omega$.
Denote by $\vec{0}$ the identity element of $\Z^\omega$, and let
$$K_n=\{e_i:i\geq n\}\cup\{\vec{0}\}.$$
For each $n\in\omega$, define
$$L_n=\bigcap_{i<n}C(\{e_i\},N_{e_i\restriction n})\cap C(K_n,N_{\vec{0}\restriction n}).$$
Clearly, $L_n$ is open in $\mathcal{T}$, and
\begin{equation*}
    L_n=\left\{\begin{pmatrix}
    I_n & O \\
    \multicolumn{2}{c}{*}
    \end{pmatrix}:\begin{pmatrix}
    I_n & O \\
    \multicolumn{2}{c}{*}
    \end{pmatrix}\in\aut(\Z^\omega)\right\}.
\end{equation*}
Here $I_n$ denotes the identity matrix of order $n$. Using block matrix multiplication, we have $L_n^2=L_n$ and $L_n^{-1}=L_n$ for each $n\in\omega$.

\begin{proposition}\label{property of autZw}
    \begin{enumerate}
        \item $(L_n)$ is an open neighborhood basis of $1_L$;
        \item $(\aut(\Z^\omega),\mathcal{T})$ is a topological group;
        \item the map $(f,x)\mapsto f(x)$ from $\aut(\Z^\omega)\times\Z^\omega$ to $\Z^\omega$ is continuous.
    \end{enumerate}
\end{proposition}

\begin{proof}
    (1) We first prove the following claim:

{\sl Claim.} For any compact $K\subseteq\Z^\omega$ and open $V\subseteq\Z^\omega$ with $K\subseteq V$, there exists an $m\in\omega$ such that $L_m\subseteq C(K,V)$.

   \begin{proof}[Proof of Claim]
        Since $K$ is compact, we can find $s_0,\dots,s_k\in\Z^{<\omega}$ such that $K\subseteq \bigcup_{i\leq k}N_{s_i}\subseteq V$. Let $m=\max\{\lh(s_i):i\leq k\}$. Then for any $f\in L_m$ and $x\in K$, we can see that
        \begin{align*}
            f(x)\restriction m=\left(\begin{pmatrix}
                I_m & O\\
                \multicolumn{2}{c}{*}
            \end{pmatrix}\begin{pmatrix}
                x\restriction m\\
                *
            \end{pmatrix}\right)\restriction m=x\restriction m.
        \end{align*}
        From $x\in K$ we have $x\in\bigcup_{i\leq k}N_{s_i}$. So $N_{x\restriction m}\subseteq\bigcup_{i\leq k}N_{s_i}$. This implies $f(x)\in\bigcup_{i\leq k}N_{s_i}\subseteq V$. Therefore, we obtain that $L_m\subseteq C(K,V)$.
    \end{proof}

    Now, for any open neighborhood $C(K_0',V_0)\cap\cdots\cap C(K_n',V_n)$ of $1_L$, by the claim above we can find an $m\in\omega$ such that
    $$L_m\subseteq C(K_0',V_0)\cap\cdots\cap C(K_n',V_n).$$
    Then by definition of $L_m$ we can see that $(L_n)$ is an open neighborhood basis of $1_L$.

    (2) For compact $K\subseteq\Z^\omega$, open $V\subseteq\Z^\omega$ and $f\in\aut(\Z^\omega)$, we can see that $C(K,V)f=C(f^{-1}(K),V)$ and $fC(K,V)=C(K,f(V))$. Note that $f^{-1}(K)$ is compact, and $f(V)$ is open. Then a straightforward verification shows that both $(L_nf)$ and $(fL_n)$ are open neighborhood bases of $f$.

    Now, let $f,g\in\aut(\Z^\omega)$ and $n\in\omega$. Suppose $g=(a_{ij})$. Since the matrix $g$ is row-finite, we can find an $M\in\omega$ such that
    $$\forall i<n\,\forall j\geq M\left(a_{ij}=0\right).$$
    Hence we can divide the matrix $g$ as
    \begin{equation*}
        g=\begin{pmatrix}
            A & O\\
            \multicolumn{2}{c}{*}
        \end{pmatrix},
    \end{equation*}
    where $A$ is a matrix of size $n\times M$. Next, we divide the matrix $g^{-1}$ as
    \begin{equation*}
        g^{-1}=\begin{pmatrix}
            B & C\\
            \multicolumn{2}{c}{*}
        \end{pmatrix},
    \end{equation*}
    where $B$ is a matrix of size $M\times n$. Then we have $AB=I_n$ and $AC=O$. By a straightforward computation, we can see that $gL_Mg^{-1}\subseteq L_n$. Thus $gL_M\subseteq L_ng$. Consequently, we obtain that
    $$(L_ng)(fL_M)^{-1}=L_ngL_Mf^{-1}=L_n(gL_Mg^{-1})gf^{-1}\subseteq L_n^2gf^{-1}=L_ngf^{-1}.$$
    Since $n$ is arbitrary, we can see that the map $(g,f)\mapsto gf^{-1}$ is continuous, which means that $(\aut(\Z^\omega),\mathcal{T})$ is a topological group.

    (3) For $n\in\omega$, let
    $$H_n=\prod_{i<n}\{0\}\times\prod_{i\geq n}\Z.$$
    Let $f\in\aut(\Z^\omega)$ and $x\in\Z^\omega$. Suppose $f=(a_{ij})$. Then for any $n\in\omega$, there exists an $M\in\omega$ such that
    $$\forall i< n\,\forall j\geq M\left(a_{ij}=0\right).$$
    Then a direct computation shows that
    \begin{align*}
        (L_nf)(x+H_M)&=\{g(y):g\in L_nf\land y\in x+H_M\}\\
        &\subseteq f(x)+H_n.
    \end{align*}
    Hence the map $(f,x)\mapsto f(x)$ is continuous.
\end{proof}

\begin{proposition}\label{autZomega is nonarchimedean Polish}
    $(\aut(\Z^\omega),\mathcal{T})$ is a non-archimedean Polish group and $(L_n)\in\dgnb(\aut(\Z^\omega))$.
\end{proposition}
\begin{proof}
    Let
    $$\dsZ=\{x\in\Z^\omega:\exists\,n\in\omega\,\forall\,i\geq n(x(i)=0)\},$$
    and endow it with the discrete topology. Note that for any $f\in\aut(\Z^\omega)$, $(L_nf)$ is an open neighborhood basis of $f$ (see the proof of Proposition~\ref{property of autZw}(2)). Therefore, we can see that $(\aut(\Z^\omega),\mathcal{T})$ is homeomorphic to $(\dsZ)^\omega$. Hence by Proposition~\ref{property of autZw}(2) we obtain that $(\aut(\Z^\omega),\mathcal{T})$ is a Polish group.

    For each $n\in\omega$, $L_n$ is an open subgroup of $\aut(\Z^\omega)$. So by Proposition~\ref{property of autZw}(1) we can see that $(\aut(\Z^\omega),\mathcal{T})$ is a non-archimedean Polish group and $(L_n)\in\dgnb(\aut(\Z^\omega))$.
\end{proof}

\begin{theorem}
    Let $G$ be a closed subgroup of $\aut(\Z^\omega)$, and let $A=G\ltimes_\rho\Z^\omega$, where $\rho(f)(x)=f(x)$ for $f\in G$ and $x\in\Z^\omega$. Then $A$ is a non-archimedean Polish group. Moreover, if $G$ is \hier\alpha\space for some $\alpha<\omega_1$, then $A$ is \hier{(1+\alpha)}.
\end{theorem}

\begin{proof}
    By Proposition~\ref{property of autZw}(3), we can see that $(f,x)\mapsto\rho(f)(x)$ is continuous. So $A$ is a non-archimedean Polish group. Since $\Z^\omega$ is abelian, it is \hier{1}.

    Let $\mathcal{G}=(G_n)=(G\cap L_n)\in\dgnb(G)$, and for each $n\in\omega$, let
    $$H_n=\prod_{i<n}\{0\}\times\prod_{i\geq n}\Z.$$
    Then $\mathcal{H}=(H_n)\in\dgnb(\Z^\omega)$. Clearly, $\mathcal{A}=(A_n)=(G_n\times H_n)\in\dgnb(A)$. We now check that the condition~(\ref{addition condition}) are satisfied. But this is clear from $\rho(f)(x+H_n)=f(x)+H_n=x+H_n$ for $f\in G_n$, $x\in\Z^\omega$ and $n\in\omega$. Consequently, if $G$ is \hier\alpha, then $A$ is \hier{(1+\alpha)}.
\end{proof}

\section{Negative Results}

We now turn to some negative results for Problem~\ref{problem of DW}.

Given a countably infinite group $\Gamma$, recall that $\aut(\Gamma)$ consists of all group automorphisms of $\Gamma$. Endow $\aut(\Gamma)$ with the pointwise convergence topology. Since
$$g\in\aut(\Gamma)\iff g\in S_\Gamma\land\forall\,\gamma_1,\gamma_2\in\Gamma\,(g(\gamma_1\gamma_2)=g(\gamma_1)g(\gamma_2)),$$
it follows that $\aut(\Gamma)$ is a closed subgroup of $S_\Gamma$. As $\Gamma$ is countably infinite, $S_\Gamma$ is isomorphic to $S_\infty$.
\begin{theorem}\label{semidirect product on countable group}
    Let $\Gamma$ be a countably infinite group, and let $G$ be a closed subgroup of $\aut(\Gamma)$. If $G$ is proper \hier\alpha, then $A=G\ltimes_\rho\Gamma$ is proper L-\hier\alpha, where $\rho:G\to\aut(\Gamma)$ is given by
    $\rho(g)(\gamma)=g(\gamma)$ for $g\in G$ and $\gamma\in\Gamma$.
\end{theorem}
\begin{proof}
    Define $\phi:A\to S_\Gamma$ as
    \begin{align*}
        \phi(g,\gamma)(\lambda)=\gamma g(\lambda)
    \end{align*}
    for $(g,\gamma)\in A$ and $\lambda\in\Gamma$.
    It is routine to check that $\phi$ is a topological group embedding.

    Enumerate $\Gamma$ as $\{\gamma_0=1_\Gamma,\gamma_1,\gamma_2,\dots\}$. For each $n\in\omega$, let
    $$G_n=\{g\in G:\forall\,i<n\,(g(\gamma_i)=\gamma_i)\}.$$
    Then $\mathcal{G}=(G_n)\in\dgnb(G)$. Define $A_0=A, A_{n+1}=G_n\times\{1_\Gamma\}$ for $n\in\omega$. Clearly, $\mathcal{A}=(A_n)\in\dgnb(A)$. Now consider the orbit tree $\otr{A}{\Gamma}$. For $p\in\omega$, we have
    \begin{align*}
        \gamma_1E_{A_{p+1}}^\Gamma\gamma_2&\iff\exists\,(g,1_\Gamma)\in A_{p+1}\left((g,1_\Gamma)\cdot\gamma_1=\gamma_2\right)\\
        &\iff\exists\,(g,1_\Gamma)\in A_{p+1}\left(g(\gamma_1)=\gamma_2\right)\\
        &\iff\exists\,g\in G_p\left(g(\gamma_1)=\gamma_2\right).
    \end{align*}
    This shows $\theta:\Gamma\to\Gamma$ defined as $\theta(\gamma)=\gamma$ is a reduction of $E_{A_{p+1}}^\Gamma$ to $E_{G_p}^\Gamma$ for any $p\in\omega$. Let $T_1=\bigcup_{n\geq 1}L_n(\otr{A}{\Gamma})$. Then by Proposition~\ref{uniform reduction induces tree embedding}, we can see that $T_1$ is isomorphic to $\otr{G}{\Gamma}$.
    Thus
    $$\rho_{\otr{A}{\Gamma}}(0,\Gamma)=\rho(T_1)=\rho(\otr{G}{\Gamma})=\rho(\mathcal{G})=\omega\cdot\alpha.$$
    Applying Lemma~\ref{another highest orbit tree}, we obtain
    $$\rho(\mathcal{A})=\rho(\otr{A}{\Gamma})=\rho_{\otr{A}{\Gamma}}(0,\Gamma)+1=\omega\cdot\alpha+1.$$
    This means that $A$ is proper L-\hier\alpha.
\end{proof}

This theorem gives a negative answer to Problem~\ref{problem of DW}.

\begin{example}\label{counterexample of ctb group}
Let
$$\Gamma=\{x\in\Z^\omega:\exists\,n\in\omega\,\forall\,i\geq n\left(x(i)=0\right)\}$$
with discrete topology. Note that this $\Gamma$ is exactly the group $\dsZ$ defined in the proof of Proposition~\ref{autZomega is nonarchimedean Polish}. For any $f\in S_\infty$, we define $\phi_f:\Gamma\to\Gamma$ as
$$\phi_f(x)=x\circ f^{-1}.$$
Then $\phi_f\in\aut(\Gamma)$. Moreover, we can verify that $f\mapsto\phi_f$ is a topological group embedding from $S_\infty$ to $\aut(\Gamma)$. Hence for any $\alpha\geq\omega$, we are able to choose a closed subgroup $G$ of $\aut(\Gamma)$ such that $G$ is proper \hier\alpha. Now according to Theorem~\ref{semidirect product on countable group}, we can see $A=G\ltimes\Gamma$ is proper L-\hier\alpha. Therefore, $A$ is not \hier{(1+\alpha)} since $1+\alpha=\alpha<\alpha+1$ for $\alpha\geq\omega$. This means that $A$ is a desired counterexample.
\end{example}

The following theorem provides a more fundamental counterexample at the lowest level for Problem~\ref{problem of DW}. It also shows that $$\rank(G)\leq\rank(N)+\rank(G/N)$$
does not hold for arbitrary non-archimedean Polish groups $G$ and their closed normal subgroups $N$. Note that in Example~\ref{counterexample of ctb group} we have $\rank(A)=\rank(G)$ and $\rank(\Gamma)=0$.

\begin{theorem}\label{3-CLI}
There exists a proper \hier{3} non-archimedean Polish group $U$ containing a closed normal subgroup $N$ such that both $N$ and $U/N$ are abelian, thus \hier{1}.
\end{theorem}

Due to the complexity and tediousness of the proof for the above theorem, we defer its proof to Section 7.

\section{Upper bound on group extensions}

In this section, our goal is to establish a general formula that bounds the complexity of $G$, in terms of the levels of $N$ and $G/N$, where $N$ is a closed normal subgroup of a non-archimedean Polish group $G$.

The following is a well-known result in group theory.

\begin{proposition}\label{basic theorem of group homomorphism}
    Suppose $G$ is a Polish group, $N$ is a closed normal subgroup of $G$, and $H$ is a clopen subgroup of $G$. Then $HN/N$ is isomorphic to $H/H\cap N$.
\end{proposition}
\begin{proof}
    Assume $\pi:G\to G/N$ is the quotient map. Then $\pi$ is surjectively open. Thus $\pi(H)$ is an open subgroup of $G/N$.

    Consider $\pi\restriction H:H\to \pi(H)$. Note that it is a surjectively and continuously open homomorphism. Also, we know that
    $$g\in\ker(\pi\restriction H)\Leftrightarrow g\in H\land\pi(g)=N\Leftrightarrow g\in H\cap N.$$
    Hence $H/H\cap N$ is isomorphic to $\pi(H)=HN/N$.
\end{proof}

For a group $G$ and its normal subgroup $N$, we denote by $[h]_N$ the left/right coset $hN=Nh$.

\begin{proposition}\label{induced action of quotient group}
    Let $H$ be a non-archimedean Polish group, $K$ a closed normal subgroup of $H$, and $Y$ a countable discrete Polish $H$-space. Let $\pi$ be the quotient map from $H$ to $H/K$. We define an action of $H/K$ on $Y/K$ by
    $$[h]_K\cdot[y]_K=[h\cdot y]_K,$$
    where $h\in H$, $y\in Y$. Then
    \begin{enumerate}
        \item this group action is well-defined and continuous;
        \item if $H^*$ is a subgroup of $H$, then for any $y \in Y$,
        \begin{align*}
            \bigcup\pi(H^*)\cdot[y]_K=H^*K\cdot y&&\itx{and}&&(H^*K\cdot y)/K=\pi(H^*)\cdot[y]_K.
        \end{align*}
    \end{enumerate}
\end{proposition}
\begin{proof}
    The proof is just a routine verification.
\end{proof}



\begin{definition}\label{definition of sxgn}
Let $G$ be a non-archimedean Polish group, and $X$ a countable discrete Polish $G$-space. For $x\in X$, $\mathcal{G}=(G_n)\in\dgnb(G)$, and a closed normal subgroup $N$ of $G$, define
$$s(x,\mathcal{G},N)=\min\{p\in\omega:G_pN\cdot x=N\cdot x\}.$$
\end{definition}
Below we verify $s$ is well-defined and give some properties of $s$.

\begin{lemma}\label{property of sxgn}
Let $G,X,x,\mathcal{G},N$ be as above, and let $\pi:G\to G/N$ be the quotient map. Then
\begin{enumerate}
    \item $s(x,\mathcal{G},N)$ is well-defined;
    \item $s(x,\mathcal{G},N)$ is the least $p\in\omega$ such that $\pi(G_p)\cdot [x]_N=\{[x]_N\}$;
    \item for any $y\in [x]_N$, we have $s(y,\mathcal{G},N)=s(x,\mathcal{G},N)$.
\end{enumerate}
\end{lemma}

\begin{proof}
    (1) We only need to prove that there exists a $p\in\omega$ such that $G_pN\cdot x=N\cdot x$. Let $\mathcal{G}'=(\pi(G_n))\in\dgnb(G/N)$. By Proposition~\ref{induced action of quotient group}, we can induce a continuous action of $G/N$ on $X/N$. By continuity, the set
    $$A=\{gN:g\in G\land gN\cdot[x]_N=[x]_N\}$$
    is an open subset of $G/N$ containing $[1_G]_N$. Thus we can find a $p\in\omega$ such that $\pi(G_p)\subseteq A$. Clearly, we have $\pi(G_p)\cdot[x]_N=\{[x]_N\}$. Applying Proposition~\ref{induced action of quotient group}(2), we obtain $G_pN\cdot x=N\cdot x$.

    (2) This follows immediately from the definition of $s(x,\mathcal{G},N)$ and Proposition~\ref{induced action of quotient group}(2).

    (3) Take any $y\in[x]_N$. Note that
    $$G_{s(y,\mathcal{G},N)}N\cdot x=G_{s(y,\mathcal{G},N)}N\cdot y=N \cdot y=N\cdot x.$$
    Hence by (2) we get $s(x,\mathcal{G},N)\leq s(y,\mathcal{G},N)$. Similarly, we have $s(y,\mathcal{G},N)\leq s(x,\mathcal{G},N)$. This gives $s(y,\mathcal{G},N)=s(x,\mathcal{G},N)$.
\end{proof}

From now on, we fix a non-archimedean Polish group $G$ and a countable discrete Polish $G$-space $X$, and we adopt the following notations and conventions:
\begin{enumerate}
    \item[(a)] $N$ is a closed normal subgroup of $G$;
    \item[(b)] the closed subgroup $N$ and the quotient group $G/N$ are both CLI groups;
    \item[(c)] let $\mathcal{G}=(G_n)\in\dgnb(G)$, and for $k\in\omega$, let $\mathcal{G}_k=(G_{k+n})_{n\in\omega}\in\dgnb(G_k)$;
    \item[(d)] let $\mathcal{N}=(N_n)=(G_n\cap N)\in\dgnb(N)$, and let $\mathcal{N}_k=(N_{k+n})_{n\in\omega}\in\dgnb(N_k)$ for each $k\in\omega$;
    \item[(e)] for any $n\in\omega$, we can see that $N_n$ is a closed normal subgroup of $G_n$. So we denote by $\pi_n$ the quotient map from $G_n$ to $G_n/N_n$. For each $k\in\omega$, define $\pi(\mathcal{G}_k)=(\pi_k(G_{k+n}))_{n\in\omega}\in\dgnb(G_k/N_k)$.
\end{enumerate}

Now for $x\in X$ and $n\in\omega$, we define
\begin{equation*}
    q_{x,n}=\begin{cases}
        0,&\itxl{if}n=0,\\
        q_{x,n-1}+\max\{s(x,\mathcal{G}_{q_{x,n-1}},N_{q_{x,n-1}}),1\},&\itxl{if}n>0.
    \end{cases}
\end{equation*}
Next, we prove some properties of $q_{x,n}$.




\begin{lemma}\label{property of qxn}
    For $x\in X$ and $n\in\omega$, we have
    \begin{enumerate}
        \item $q_{x,n+1}>q_{x,n}$;
        \item $G_{q_{x,n+1}}N_{q_{x,n}}\cdot x=N_{q_{x,n}}\cdot x$;
        \item if $G_{q_{x,n}}N_{q_{x,n}}\cdot x\neq N_{q_{x,n}}\cdot x$, then
        \begin{equation*}
            q_{x,n+1}=\itxl{the least}p\in\omega\itx{such that}\pi(G_p)\cdot[x]_{N_{q_{x,n}}}=\{[x]_{N_{q_{x,n}}}\};
        \end{equation*}
        \item for any $y\in N_{q_{x,n}}\cdot x$ and $i\leq n+1$, $q_{y,i}=q_{x,i}$ holds.
    \end{enumerate}
\end{lemma}

\begin{proof}
    (1) can be derived from the definition of $q_{x,n}$ directly. For (2), just note that $G_{q_{x,n}+s(x,\mathcal{G}_{q_{x,n}},N_{q_{x,n}})}N_{q_{x,n}}\cdot x=N_{q_{x,n}}\cdot x$ by Lemma~\ref{property of sxgn}. To prove (3), note that $G_{q_{x,n}}N_{q_{x,n}}\cdot x\neq N_{q_{x,n}}\cdot x$ could imply
    $$s(x,\mathcal{G}_{q_{x,n}},N_{q_{x,n}})>0$$
    by definition of $s(x,\mathcal{G}_{q_{x,n}},N_{q_{x,n}})$. So $q_{x,n+1}=q_{x,n}+s(x,\mathcal{G}_{q_{x,n}},N_{q_{x,n}})$. This leads to (3) by Lemma~\ref{property of sxgn}(2).

    Finally, we prove (4) by induction on $i\leq n+1$. For $i=0$, clearly $q_{y,0}=0=q_{x,0}$. Suppose for $i-1$ we have $q_{y,i-1}=q_{x,i-1}$. Since $y\in N_{q_{x,n}}\cdot x\subseteq N_{q_{x,i-1}}\cdot x$, by Lemma~\ref{property of sxgn}(3), we have
    $$s(y,\mathcal{G}_{q_{y,i-1}},N_{q_{y,i-1}})=s(x,\mathcal{G}_{q_{x,i-1}},N_{q_{x,i-1}}).$$
    Then from the definitions of $q_{y,i}$ and $q_{x,i}$, we obtain $q_{y,i}=q_{x,i}$. This finishes the induction.
\end{proof}

To proceed, we introduce the following auxiliary definition.
\begin{definition}
    Let $G$ be a non-archimedean Polish group, $\mathcal{G}=(G_n)\in\dgnb(G)$, and $X$ a countable discrete Polish $G$-space. We define $(T_\mathcal{G}^X)^+=T_\mathcal{G}^X\cup\{(0,G\cdot x):x\in X\}$.
\end{definition}
In this definition, we add all singleton $G$-orbits to the orbit tree. This will be used in the proof of Lemma~\ref{psix is order preserving}. Note that $\rho(\otr{G}{X})=\rho((\otr{G}{X})^+)$ if $\otr{G}{X}$ is nonempty.

For a binary relation $R$ on a set $A$, recall that $R$ is {\it well-founded} if each nonempty subset of $A$ has a $R$-minimal element. Given a well-founded relation $R$ on $A$, for each $x\in A$, we can define the {\it rank} of $x$ in $A$ as
$$\rho_R(x)=\sup\{\rho_R(y)+1:y\in A\land y\,R\,x\}.$$
Note that $x\in A$ is a $R$-minimal element iff $\rho_R(x)=0$. Next we define the {\it rank} of $R$ by
$$\rho(R)=\sup\{\rho_R(x)+1:x\in A\}.$$
These notions coincide with previous definitions of rank for trees, if we treat a well-founded tree $(T,<)$ as a well-founded relation $(T,>)$.

Let $R,S$ be two binary relations on $A,B$, respectively. A function $f:A\to B$ is said to be a {\it homomorphism} from $(A,R)$ to $(B,S)$ if
$$\forall\,x_1,x_2\in A\left(x_1\,R\,x_2\implies f(x_1)\,S\,f(x_2)\right).$$
If such a homomorphism $f$ exists and $S$ is well-founded, then $R$ is also well-founded. Moreover, by induction one can show that $\rho(R)\leq\rho(S)$.

Before diving into the details, we briefly outline our approach. We begin by defining a tree $(T_B,<_{T_B})$. Next, we replace each node in $T_B$ with a new tree, and lexicographically order the resulting set. Denoting the resulting set and order by $A$ and $R$, respectively, we show that $(A,R)$ is well-founded. We then construct a homomorphism from $(\otr{G}{X},>)$ to $(A,R)$. It follows that the rank of $R$ is an upper bound for the rank of $\otr{G}{X}$.

Now, we define a tree
$$T_B=\{(q_{x,n},N_{q_{x,n}}\cdot x)\in \otr{N}{X}:x\in X\land n\in\omega\}\cup\{\emptyset\},$$
which is ordered by
$$C<_{T_B}D\iff(C=\emptyset\land D\neq\emptyset)\lor(C<D).$$
Here $C,D\in T_B$, and $<$ is the order of the tree $\otr{N}{X}$. Clearly, $\emptyset$ is the $<_{T_B}$-least element in $T_B$, and thus from~\cite[Corollary 3.12]{DingW},
$$\rho(T_B)=\rho_{T_B}(\emptyset)+1\le\rho(\otr{N}{X})+1\le\rho(\mathcal N)+1.$$

Next we define a function $\Phi$ on $T_B$ as
\begin{align*}
    \Phi(\emptyset)&=(T_{\pi(\mathcal{G})}^{X/N})^+,\\
    \Phi(q_{x,n},N_{q_{x,n}}\cdot x)&=(T_{\pi(\mathcal{G}_{q_{x,n+1}})}^{N_{q_{x,n}}\cdot x/N_{q_{x,n+1}}})^+
\end{align*}
for $x\in X$ and $n\in\omega$. For any $x,y\in X$ and $n\in\omega$, when $(q_{x,n},N_{q_{x,n}}\cdot x)=(q_{y,n},N_{q_{y,n}}\cdot y)$,  we can see that $q_{x,n+1}=q_{y,n+1}$ by Lemma~\ref{property of qxn}(4). Therefore,
$$(T_{\pi(\mathcal{G}_{q_{x,n+1}})}^{N_{q_{x,n}}\cdot x/N_{q_{x,n+1}}})^+=(T_{\pi(\mathcal{G}_{q_{y,n+1}})}^{N_{q_{y,n}}\cdot y/N_{q_{y,n+1}}})^+.$$
This shows that $\Phi$ is well-defined.

Let
$$A=\bigcup_{s\in T_B}\{s\}\times\Phi(s).$$
For $s\in T_B$, we denote by $<_{\Phi(s)}$ the order of the tree $\Phi(s)$. We define a relation $R$ on $A$ by
$$(t,D)\,R\,(s,C)\iff s<_{T_B}t\lor(s=t\land C<_{\Phi(t)}D),$$
where $(s,C),(t,D)\in A$.
\begin{lemma}\label{upper bound on R}
    $R$ is a well-founded relation on $A$. Furthermore, we have
    $$\rho(R)\leq\sup\{\rho(\Phi(s)):s\in T_B\}\cdot\rho(T_B).$$
\end{lemma}
\begin{proof}
    By Proposition~\ref{basic theorem of group homomorphism}, for any $n\in\omega$, $G_n/N_n$ is isomorphic to $G_nN/N$, and $G_nN/N=\pi_0(G_n)$ is an open subgroup of $G/N$. Since $G/N$ is CLI, $G_n/N_n$ is also CLI. Thus by Theorem~\ref{relationship between CLI groups and orbit trees}, for each $s\in T_B$, $\Phi(s)$ is a well-founded tree. Similarly, $T_B$ is a well-founded tree because $N$ is a CLI group. Then it is trivial to see that $R$ is well-founded.

    Put $\beta=\sup\{\rho(\Phi(s)):s\in T_B\}$. Then $\beta>\rho_{\Phi(s)}(C)$ for all $(s,C)\in A$. Next, we will prove
    $$\rho_{R}(s,C)\leq\beta\cdot\rho_{T_B}(s)+\rho_{\Phi(s)}(C)$$
    by induction on $\rho_R(s,C)$ for $(s,C)\in A$.
    If $\rho_{R}(s,C)=0$, then the inequality clearly holds. Now suppose the inequality holds for any $(t,D)\in A$ with $\rho_{R}(t,D)<\rho_{R}(s,C)$. Then
    \begin{align*}
        \rho_{R}(s,C)&=\sup\{\rho_{R}(t,D)+1:(t,D)\,R\,(s,C)\} \\
        &\leq\sup\{\beta\cdot\rho_{T_B}(t)+\rho_{\Phi(t)}(D)+1:s<_{T_B}t\lor(s=t\land C<_{\Phi(t)}D)\} \\
        &=\sup(\{\beta\cdot\rho_{T_B}(t)+\rho_{\Phi(t)}(D)+1:s<_{T_B}t\land D\in\Phi(t)\} \\
        &\qquad\qquad\cup\{\beta\cdot\rho_{T_B}(s)+\rho_{\Phi(s)}(D)+1:C<_{\Phi(s)}D\}) \\
        &\leq\sup\{\beta\cdot\rho_{T_B}(s)+\rho_{\Phi(s)}(D)+1:C<_{\Phi(s)}D\}\\
        &=\beta\cdot\rho_{T_B}(s)+\rho_{\Phi(s)}(C).
    \end{align*}
    This finishes the induction.

    Finally, from the inequality above we can see that
    \begin{align*}
        \rho(R)&=\sup\{\rho(s,C)+1:(s,C)\in A\}\\
        &\leq\sup\{\beta\cdot\rho_{T_B}(s)+\rho_{\Phi(s)}(C)+1:(s,C)\in A\}\\
        &\leq\beta\cdot\rho(T_B).
    \end{align*}

\end{proof}
For $x\in X$, let $r_x=\max\{n\in\omega:G_n\cdot x\ne\{x\}\}$ and
$$I_x=\{(n,G_n\cdot x):n\in\omega\land G_n\cdot x\ne\{x\}\}.$$
We define a function $\psi_x$ on $I_x$ as follows: for $(n,G_n\cdot x)\in I_x$,
\begin{enumerate}
    \item[(i)] if $n\in[0,q_{x,1})$, then
    $$\psi_x(n,G_n\cdot x)=(\emptyset,(n,\pi_0(G_n)\cdot[x]_N));$$
    \item[(ii)] if there exists an $i>0$ such that $n\in[q_{x,i},q_{x,i+1})$, then
    $$\psi_x(n,G_n\cdot x)=((q_{x,i-1},N_{q_{x,i-1}}\cdot x),(n-q_{x,i},\pi_{q_{x,i}}(G_n)\cdot[x]_{N_{q_{x,i}}})).$$
\end{enumerate}

\begin{lemma}\label{psix is order preserving}
    For any $x\in X$, the image of $\psi_x$ is contained in $A$. Moreover, for any $(m,G_m\cdot x),(n,G_n\cdot x)\in I_x$, if $(m,G_m\cdot x)<(n,G_n\cdot x)$, then we have $\psi_x(n,G_n\cdot x)\,R\,\psi_x(m,G_m\cdot x)$.
\end{lemma}

\begin{proof}
    Take any $(n,G_n\cdot x)\in I_x$. Since $(q_{x,i})$ is strictly increasing and $q_{x,0}=0$, there exists a unique $i\in\omega$ such that $n\in[q_{x,i},q_{x,i+1})$. Now if $i>0$, we consider the following two cases:

    Case 1: If $G_{q_{x,i}}N_{q_{x,i}}\cdot x\ne N_{q_{x,i}}\cdot x$, then Lemma~\ref{property of qxn}(3) yields
    $$(n-q_{x,i},\pi_{q_{x,i}}(G_n)\cdot[x]_{N_{q_{x,i}}})\in T_{\pi(\mathcal{G}_{q_{x,i}})}^{N_{q_{x,i-1}}\cdot x/N_{q_{x,i}}}\subseteq\Phi(q_{x,i-1},[x]_{N_{q_{x,i-1}}}).$$

    Case 2: If $G_{q_{x,i}}N_{q_{x,i}}\cdot x=N_{q_{x,i}}\cdot x$, we have $q_{x,i+1}=q_{x,i}+1$. This implies $n=q_{x,i}$, hence
    $$(n-q_{x,i},\pi_{q_{x,i}}(G_n)\cdot[x]_{N_{q_{x,i}}})\in(T_{\pi(G_{q_{x,i}})}^{N_{q_{x,i-1}}\cdot x/N_{q_{x,i}}})^+=\Phi(q_{x,i-1},[x]_{N_{q_{x,i-1}}}).$$
    Thus $\psi_x(n,G_n\cdot x)$ lies in $A$. The same argument applies for $i=0$. This shows that the image of $\psi_x$ is contained in $A$.

    For any $(m,G_m\cdot x)<(n,G_n\cdot x)$ in $I_x$, we can find unique $i,j\in\omega$ such that $m\in[q_{x,i},q_{x,i+1})$ and $n\in[q_{x,j},q_{x,j+1})$. Clearly, $i\le j$. We consider the following cases:

    Case 1: If $0<i<j$, then $q_{x,i-1}<q_{x,j-1}$. Thus
    $$(q_{x,i-1},N_{q_{x,i-1}}\cdot x)<_{T_B}(q_{x,j-1},N_{q_{x,j-1}}\cdot x).$$
    Then we have $\psi_x(n,G_n\cdot x)\,R\,\psi_x(m,G_m\cdot x)$.

    Case 2: If $0<i=j$, then $m-q_{x,i}<n-q_{x,i}$. So
    $$(m-q_{x,i},\pi_{q_{x,i}}(G_m)\cdot[x]_{N_{q_{x,i}}})<_{\Phi(s)}(n-q_{x,i},\pi_{q_{x,i}}(G_n)\cdot[x]_{N_{q_{x,i}}}),$$
    where $s=(q_{x,i-1},N_{q_{x,i-1}}\cdot x)$. Thus $\psi_x(n,G_n\cdot x)\,R\,\psi_x(m,G_m\cdot x)$ also holds.

    The remaining cases $0=i<j$ and $0=i=j$ can be treated similarly, in each case we can obtain
    $$\psi_x(n,G_n\cdot x)\,R\,\psi_x(m,G_m\cdot x).$$
    This finishes the proof.
\end{proof}

Let $\mathcal{I}=\{I_x:x\in X\}$ and $\mathcal{F}=\{\psi_x:x\in X\}$. Then it is easy to see that $\bigcup\mathcal{I}=\otr{G}{X}$.

\begin{lemma}\label{family of psix are compatible}
For $x,y\in X$, if $(n,C)\in I_x\cap I_y$, then $\psi_x(n,C)=\psi_y(n,C)$.
\end{lemma}

\begin{proof}
Take any $x,y\in X$, without loss of generality assume $I_x\cap I_y\neq\emptyset$. Note that
$$I_x=\{(0,G_0\cdot x),(1,G_1\cdot x),\ldots,(r_x,G_{r_x}\cdot x)\},$$
$$I_y=\{(0,G_0\cdot y),(1,G_1\cdot y),\ldots,(r_y,G_{r_y}\cdot y)\}.$$
Let $r=\min\{r_x,r_y\}$, then $G_n\cdot x=G_n\cdot y$ for $n=0,1,\ldots,r$, and thus
$$I_x\cap I_y=\{(0,G_0\cdot x),(1,G_1\cdot x),\ldots,(r,G_r\cdot x)\}.$$
Take any $(m,G_m\cdot x)\in I_x\cap I_y$, then there is a unique $i\in\omega$ such that $m\in[q_{x,i},q_{x,i+1})$.

If $i=0$, since $N$ is a normal subgroup, from $G_m\cdot x=G_m\cdot y$ we obtain
$$G_mN\cdot x=NG_m\cdot x=NG_m\cdot y=G_mN\cdot y.$$
Therefore, by Proposition~\ref{induced action of quotient group} we have $\pi_0(G_m)\cdot[x]_N=\pi_0(G_m)\cdot[y]_N$. This implies $\psi_x(m,G_m\cdot x)=\psi_y(m,G_m\cdot y)$.

    If $i>0$, then by Lemma~\ref{property of qxn}(2) we get
    $$N_{q_{x,i-1}}\cdot x=G_{q_{x,i}}N_{q_{x,i-1}}\cdot x\supseteq G_{q_{x,i}}\cdot x=G_{q_{x,i}}\cdot y.$$
    Hence $y\in N_{q_{x,i-1}}\cdot x$. So using Lemma~\ref{property of qxn}(4), we obtain $q_{x,i-1}=q_{y,i-1}$ and $q_{x,i}=q_{y,i}$. This implies $N_{q_{y,i-1}}\cdot y=N_{q_{x,i-1}}\cdot x$. On the other hand, since $G_m\cdot x=G_m\cdot y$ and $N_{q_{x,i}}=N_{q_{y,i}}$ is a normal subgroup of $G_{q_{x,i}}=G_{q_{y,i}}$, we have $G_mN_{q_{x,i}}\cdot x=G_mN_{q_{y,i}}\cdot y$. By Proposition~\ref{induced action of quotient group}, we can imply that
    $$\pi_{q_{x,i}}(G_m)\cdot[x]_{N_{q_{x,i}}}=\pi_{q_{y,i}}(G_m)\cdot[y]_{N_{q_{y,i}}}.$$
    Thus $\psi_x(m,G_m\cdot x)=\psi_y(m,G_m\cdot y)$ also holds.
\end{proof}

Now we are ready to estimate the rank of $\otr{G}{X}$.
\begin{lemma}\label{upper bound on rank of orbit tree}
    If $G/N$ is nontrivial, i.e., $G\neq N$, then
    $$\rho(\otr{G}{X})\leq\rho(\pi(\mathcal{G}))\cdot(\rho(\mathcal{N})+1).$$
\end{lemma}
\begin{proof}
Let $\Psi=\bigcup\mathcal{F}$. From Lemma~\ref{family of psix are compatible} we can see that $\Psi$ maps from $\otr{G}{X}$ to $A$.
Fix two elements in $\otr{G}{X}$, say, $(m,G_m\cdot x)\in I_x$ and $(n,G_n\cdot y)\in I_y$. Without loss of generality, we can assume that $r_x\ge r_y$,
then $(n,G_n\cdot y)=(n,G_n\cdot x)\in I_x$. By Lemma~\ref{psix is order preserving}, $(m,G_m\cdot x)<(n,G_n\cdot x)$ implies $\Psi(n,G_n\cdot x)\,R\,\Psi(m,G_m\cdot x)$,
so $\Psi$ is a homomorphism from $(\otr{G}{X},>)$ to $(A,R)$. Applying Lemma~\ref{upper bound on R}, we obtain
    \begin{align*}
        \rho(\otr{G}{X})\leq\rho(R)\leq\sup\{\rho(\Phi(s)):s\in T_B\}\cdot\rho(T_B).
    \end{align*}

    It is clear that $\rho(T_B)\leq\rho(\mathcal{N})+1$. Now for $\emptyset\in T_B$, from~\cite[Corollary 3.12]{DingW} we obtain
    \begin{align*}
        &\rho(\Phi(\emptyset))=\rho((T_{\pi(\mathcal{G})}^{X/N})^+)\leq\rho(\pi(\mathcal{G})).
    \end{align*}
    For each $(q_{x,n},N_{q_{x,n}}\cdot x)\in T_B$, we also obtain from~\cite[Corollary 3.12]{DingW} that
    \begin{align*}
        &\rho(\Phi(q_{x,n},N_{q_{x,n}}\cdot x))=\rho((T_{\pi(\mathcal{G}_{q_{x,n+1}})}^{N_{q_{x,n}}\cdot x/N_{q_{x,n+1}}})^+)\leq\rho(\pi(\mathcal{G}_{q_{x,n+1}})).
    \end{align*}
    Note that $G_{q_{x,n+1}}/N_{q_{x,n+1}}$ can be viewed as a clopen subgroup of $G/N$ by Proposition~\ref{basic theorem of group homomorphism}. So $\rho(\pi(\mathcal{G}_{q_{x,n+1}}))\leq\rho(\mathcal{G})$. Therefore, we have
    $$\rho(\otr{G}{X})\leq\rho(\pi(\mathcal{G}))\cdot(\rho(\mathcal{N})+1)$$
as desired.
\end{proof}

In the preceding lemma, if we take $X=X(\mathcal{G})$, then
$$\rho(\mathcal{G})\leq\rho(\pi(\mathcal{G}))\cdot(\rho(\mathcal{N})+1).$$
Consequently, we can give an upper bound for the complexity of $G$ in terms of the levels of $G/N$ and $N$.
\begin{theorem}
    Let $G$ be a non-archimedean Polish group, and let $N$ be a proper closed normal subgroup of $G$. If $N$ is \hier\alpha\space and $G/N$ is \hier\beta\space for some $\alpha,\beta<\omega_1$, then $G$ is \hier{\beta\cdot(\omega\cdot\alpha+1)}.
\end{theorem}

We point out that this upper bound is rather large. For instance, even if both $N$ and $G/N$ are \hier1\space(equivalently, TSI), the rank of $G$ may still be as large as $\omega+1$. However, we do not currently know how to construct such $G$ and $N$ to attain $\omega+1$.

Then the following question arises naturally.

\begin{question}
If $N$ is \hier\alpha\space and $G/N$ is \hier\beta\space for some $\alpha,\beta<\omega_1$, is the upper bound \hier{\beta\cdot(\omega\cdot\alpha+1)} given in the above theorem optimal?
\end{question}

\section{Proof of Theorem~\ref{3-CLI}}

In this section, we will construct a proper \hier3\space group $U$, which has a closed normal subgroup $N$ such that both $N$ and $U/N$ are abelian, hence \hier1.

Let $G=\Z^\omega$ and $G_n=\prod_{i<n}\{0\}\times\prod_{i\geq n}\Z$. Define
$$L=\{f:f\mbox{ is a continuous homomorphism from $G$ to $G$}\}.$$
Then each $f\in L$ can be identified with an infinite matrix $A_f=(a_{ij})$ over $\Z$ which is row-finite, i.e., for every $i\in\omega$ there exists an $N\in\omega$ such that $a_{ij}=0$ for all $j\geq N$.

For $n\in\omega$, let $L_n$ denote the set of all infinite matrices in $L$ where the first $n$ rows of entries are $0$.
We can write any element of $L_n$ as a matrix of the form $\begin{pmatrix}O_n & O\\ * & *\end{pmatrix}$, where $O_n$ stands for the zero matrix of order $n$.
Then there is a group topology $\mathcal{T}$ on $(L,+)$ such that $(L_n)$ forms a neighborhood basis of $1_L$.
We mention the following proposition, which will be useful to construct our group. Its proof will be given after Lemma~\ref{N2N3的性质}.
\begin{lemma}\label{矩阵乘法和线性变换作用的连续性}
    \begin{enumerate}
        \item[(1)] The matrix multiplication on $L$ is continuous;
        \item[(2)] the function $(a,x)\mapsto ax$ from $L\times\Z^\omega$ to $\Z^\omega$ is continuous.
    \end{enumerate}
\end{lemma}

Now we have the following lemma.
\begin{lemma}
    $(L,+,\mathcal{T})$ is a non-archimedean Polish group.
\end{lemma}
\begin{proof}
We only need to show that $(L,+,\mathcal{T})$ is Polish, then it is clearly a non-archimedean Polish group.

    For $a=(a_{ij})$ and $b=(b_{ij})$ in $L$, define
    \begin{align*}
        d(a,b)=\begin{cases}
            2^{-n},&\itx{if}n\itx{is the least number such that}\\
            &\exists\,j\in\omega\left(a_{nj}\neq b_{nj}\right),\\
            0,&\itx{if there is no such}n.
        \end{cases}
    \end{align*}
    It is straightforward to verify that $d$ is a complete ultrametric on $L$. Since $B(O,2^{-(n-1)})=\{a\in L:d(a,O)<2^{-(n-1)}\}=L_n$ for each $n\in\omega$, we can see that $d$ is compatible with $\mathcal{T}$. It follows that $L$ is a Polish group.
\end{proof}


Next, we define
\begin{align*}
    U=\Bigg\{
        \begin{pmatrix}
            1 & a & d & f\\
            0 & 1 & b & e\\
            0 & 0 & 1 & c\\
            0 & 0 & 0 & 1
        \end{pmatrix}:c,e,f\in\Z^\omega\wedge a,b,d\in L
    \Bigg\},
\end{align*}
and endow $U$ with the product topology of $L^3\times G^3$. Take
\begin{equation}\label{U中矩阵的符号}
    A=\begin{pmatrix}
        1 & a & d & f\\
        0 & 1 & b & e\\
        0 & 0 & 1 & c\\
        0 & 0 & 0 & 1
    \end{pmatrix},
    X=\begin{pmatrix}
        1 & u & x & z\\
        0 & 1 & v & y\\
        0 & 0 & 1 & w\\
        0 & 0 & 0 & 1
    \end{pmatrix}.
\end{equation}
From now on, when we use the symbol $X$ to denote a matrix in $U$, we use the letters $u,v,w,x,y,z$ to denote its entries, as in \eqref{U中矩阵的符号}. The same convention also applies to the symbol $A$. Moreover, when we use subscripted or superscripted versions of  $X$ and $A$, we also use the corresponding subscripted or superscripted versions of letters $u,v,w,\dots$ and $a,b,c,\dots$ to denote their entries.

The multiplication of $U$ is defined as usual matrix multiplication:
\begin{align*}
    AX=\begin{pmatrix}
        1 & a+u & a v + d + x & a y + d w + f + z\\
        0 & 1 & b+v & b w + e + y\\
        0 & 0 & 1 & c + w\\
        0 & 0 & 0 & 1
    \end{pmatrix}.
\end{align*}
By a straightforward verification, we can show that $U$ with this multiplication forms a group. Strictly speaking, one could define $U$  as the set of $6$-tuples, and define the multiplication accordingly. However, this representation is cumbersome, so we adopt the matrix representation instead.

Now let
\begin{align*}
    U_n=\left\{
        \begin{pmatrix}
            1 & a & d & f\\
            0 & 1 & b & e\\
            0 & 0 & 1 & c\\
            0 & 0 & 0 & 1
        \end{pmatrix}:c,e,f\in G_n\wedge a,b,d\in L_n\right\}
\end{align*}
for $n\in\omega$. We can obtain the following proposition:
\begin{proposition}\label{U是non-archimedean Polish群}
    $U$ is a non-archimedean Polish group and $\mathcal{U}=(U_n)\in\dgnb(U)$.
\end{proposition}

\begin{proof}
    Choose any $A,X\in U$. Note that
    \begin{align*}
        AX&=\begin{pmatrix}
            1 & a+u & a v + d + x & a y + d w + f + z\\
            0 & 1 & b+v & b w + e + y\\
            0 & 0 & 1 & c + w\\
            0 & 0 & 0 & 1
        \end{pmatrix},\\
        X^{-1}&=\begin{pmatrix}
            1 & -u & u v - x & -u v w + u y + x w - z\\
            0 & 1 & -v & v w - y\\
            0 & 0 & 1 & -w\\
            0 & 0 & 0 & 1
        \end{pmatrix}.
    \end{align*}
Therefore, the multiplication and inverse operation over $U$ are both continuous by Lemma~\ref{矩阵乘法和线性变换作用的连续性}. Hence $U$ is a topological group.
Moreover, $U$ is a non-archimedean Polish group since the product topology on $U$ is Polish. The rest of the proposition follows easily.
\end{proof}

For $X_1,X_2\in U$, we have
\begin{align}\label{一般情形下的Un陪集}
    &X_2^{-1} X_1 =
    \begin{pmatrix}
    1 & u_1 - u_2 &
    x_1 - x_2 + u_2(v_2 - v_1) &
    \Gamma
    \\
    0 & 1 & v_1 - v_2 &
    y_1 - y_2 + v_2(w_2 - w_1)
    \\
    0 & 0 & 1 & w_1 - w_2
    \\
    0 & 0 & 0 & 1
    \end{pmatrix},
\end{align}
where $\Gamma=z_1-z_2+u_2(y_2-y_1)+u_2v_2(w_1-w_2)+x_2(w_2-w_1)$. Specifically, when $u_1=u_2,v_1=v_2$ and $w_1=w_2$, we have
\begin{equation}\label{特殊情形下的U_n陪集}
    X_2^{-1}X_1=\begin{pmatrix}
    1 & 0 & x_1 - x_2 & z_1 - z_2 + u_2 (y_2 - y_1) \\
    0 & 1 & 0 & y_1 - y_2 \\
    0 & 0 & 1 & 0 \\
    0 & 0 & 0 & 1
    \end{pmatrix}.
\end{equation}
Now for $a\in L$, we define
\begin{align*}
    &R_i(a)=\min\{j\in\omega:\forall\,j'\geq j\,(a_{ij'}=0)\},\\
    &R^n(a)=\max\{R_0(a),R_1(a),\dots,R_{n-1}(a)\}.
\end{align*}
for $i,n\in\omega$. Fix an $n\geq 1$. For each $X\in U$, we define
\begin{align*}
    N_1(X)&=n,\\
    N_2(X)&=\max\{N_1(X),R^n(u)\},\\
    N_3(X)&=\max\{N_2(X),R^n(uv-x),R^n(v)\}.
\end{align*}
The function $N_2$ and $N_3$ have the following properties:
\begin{lemma}\label{N2N3的性质}
    \begin{enumerate}
        \item[(1)] if $X_1U_n=X_2U_n$, then $N_i(X_1)=N_i(X_2)$ for $i=1,2,3$.
        \item[(2)] $N_3(X)$ is the least $N\in\omega$ such that $U_N\cdot XU_n=\{XU_n\}$;
        \item[(3)] for any $X'U_n\in U_{N_2(X)}\cdot XU_n$, $N_3(X')=N_3(X)$ holds;
        \item[(4)] for any $X'U_n\in U_n\cdot XU_n$, $N_2(X')=N_2(X)$ holds.
    \end{enumerate}
\end{lemma}
\begin{proof}
    For $X\in U$ and $p\in\omega$, first we present the following observations:
    \begin{enumerate}
        \item[(i)] if $p\geq n$, then $L_pv\subseteq L_n$. To see this, we just use block matrix multiplication
        \begin{align*}
            &\begin{pmatrix}
                O_p & O \\
                A & B
            \end{pmatrix}\begin{pmatrix}
                C & D \\
                E & F
            \end{pmatrix}=\begin{pmatrix}
                O_p & O\\
                * & *
            \end{pmatrix}\in L_p\subseteq L_n.
        \end{align*}
        \item[(ii)] if $p\geq n$, then $uL_p\subseteq L_n\iff p\geq R^n(u)$. This is also a direct application of block matrix multiplication
        \begin{align*}
            &\begin{pmatrix}
                A_{00} & A_{01} & A_{02}\\
                A_{10} & A_{11} & A_{12}\\
                A_{20} & A_{21} & A_{22}
            \end{pmatrix}\begin{pmatrix}
                O_n & O & O\\
                O & O_{p-n} & O\\
                B & C & D
            \end{pmatrix}=\begin{pmatrix}
                A_{02}B & A_{02}C & A_{02}D\\
                A_{12}B & A_{12}C & A_{12}D\\
                A_{22}B & A_{22}C & A_{22}D
            \end{pmatrix}.
        \end{align*}
        \item[(iii)] $vG_p\subseteq G_n\iff p\geq R^n(v)$.
        \item[(iv)] if $w\neq(0,0,\dots)$, then $uL_pw\subseteq G_n\iff p\geq R^n(u)$. To see this, note that
        $$\prod_{i<p}\{0\}\times\prod_{i\geq p}(k\mathbb{Z})\subseteq L_pw\subseteq G_p,$$
        where $k=\min\{|w(j)|:j\in\omega\land w(j)\neq 0\}$.
    \end{enumerate}

To prove clause (1), suppose $X_1U_n=X_2U_n$. Since $X_2^{-1}X_1\in U_n$, we can see that $u_1-u_2\in L_n$ and $v_1-v_2\in L_n$ from (\ref{一般情形下的Un陪集}). Hence we have $R^n(u_1)=R^n(u_2)$ and $R^n(v_1)=R^n(v_2)$. So $N_2(X_1)=N_2(X_2)$. We also have $x_1-x_2+u_2(v_2-v_1)\in L_n$ by $X_2^{-1}X_1\in U_n$.
Note that $R^n((u_2-u_1)v_1)=0$. All these together imply
    \begin{align*}
        0&=R^n(x_1-x_2+u_2(v_2-v_1))\\
        &=R^n(x_1-x_2+u_2(v_2-v_1)+(u_2-u_1)v_1)\\
        &=R^n(u_2v_2-x_2-(u_1v_1-x_1)).
    \end{align*}
This shows $R^n(u_2v_2-x_2)=R^n(u_1v_1-x_1)$. So $N_3(X_2)=N_3(X_1)$.


Before proving clauses (2), (3), and (4), we do some matrix calculations. For $A,X\in U$, we define
    \begin{align*}
        F_1(A,X)&=\begin{pmatrix}
            1 & u & - u b + x & -ubw-ue+ubc-xc+z\\
            0 & 1 & v & y - v c\\
            0 & 0 & 1 & w\\
            0 & 0 & 0 & 1
        \end{pmatrix}
    \end{align*}
    We claim that for $p\in\omega$,
    \begin{equation}\label{矩阵等式不改变uvw的等价条件}
        p\geq n\iff\forall\,A\in U_p\,(AXU_n=F_1(A,X)U_n).
    \end{equation}
    To prove the right implication, we construct an equation
    \begin{align*}
        AXU_n&=\begin{pmatrix}
            1 & a+u & a v + d + x & a y + d w + f + z\\
            0 & 1 & b+v & b w + e + y\\
            0 & 0 & 1 & c + w\\
            0 & 0 & 0 & 1
        \end{pmatrix}U_n\\
        &=\begin{pmatrix}
            1 & u & \gamma_1 & \gamma_3\\
            0 & 1 & v & \gamma_2\\
            0 & 0 & 1 & w\\
            0 & 0 & 0 & 1
        \end{pmatrix}U_n
    \end{align*}
    for $A\in U_p$. Using (\ref{一般情形下的Un陪集}) and previous observations, we derive
    \begin{align*}
        \gamma_1+L_n&=av-ub+d+x+L_n=-ub+x+L_n,\\
        \gamma_2+G_n&=bw+e+y-vc+G_n=y-vc+G_n,\\
        \gamma_3+G_n&=ay+dw+f+z+u(\gamma_2-bw-e-y)+uvc-\gamma_1c+G_n\\
        &=z+u(\gamma_2-bw-e-y+vc)+ubc-xc+G_n.
    \end{align*}
    Hence
    \begin{align*}
        \gamma_1&=-ub+x,\\
        \gamma_2&=y-vc,\\
        \gamma_3&=z-ubw-ue+ubc-xc
    \end{align*}
    is one solution to our equation. This shows $AXU_n=F_1(A,X)U_n$.

    To prove the left implication, just note that
    \begin{align*}
        &\forall\,A\in U_p\left(AXU_n=F_1(A,X)U_n\right)\\
        \implies&\forall\,A\in U_p\left(a\in L_n\right)\\
        \implies&p\geq n.
    \end{align*}

    Next we define
    \begin{align*}
        F_2(A,X)&=\begin{pmatrix}
            1 & u & x & -xc+z\\
            0 & 1 & v & y - v c\\
            0 & 0 & 1 & w\\
            0 & 0 & 0 & 1
        \end{pmatrix} ,\\
        F_3(A,X)&=\begin{pmatrix}
            1 & u & x & (uv-x)c+z\\
            0 & 1 & v & y\\
            0 & 0 & 1 & w\\
            0 & 0 & 0 & 1
        \end{pmatrix}
    \end{align*}
    for $A,X\in U$. Then using (\ref{特殊情形下的U_n陪集}) and similar method as before, we obtain
    \begin{align*}
        &p\geq N_2(X)\\
        \implies&p\geq N_2(X)\land\forall\,A\in U_p\left(AXU_n=F_1(A,X)U_n\right)\\
        \implies&\forall\,A\in U_p\left(AXU_n=F_2(A,X)U_n\right)\\
        \implies&\forall\,A\in U_p\left(a\in L_n\land AXU_n=F_2(A,X)U_n\right)\\
        \implies&p\geq n\land\forall\,A\in U_p\left(F_1(A,X)U_n=F_2(A,X)U_n\right)\\
        \implies&p\geq n\land\forall\,A\in U_p\left(-ub\in L_n\right)\\
        \implies&p\geq N_2(X).
    \end{align*}
    In conclusion, we have
    \begin{equation}\label{矩阵等式不改变uvwx的等价条件}
        p\geq N_2(X)\iff\forall\,A\in U_p\left(AXU_n=F_2(A,X)U_n\right).
    \end{equation}
    Analogously, we have
    \begin{align}
        &\begin{aligned}
            &p\geq N_2(X)\land p\geq R^n(v)\\
            \implies&\forall\,A\in U_p\left(AXU_n=F_3(A,X)U_n\right)\\
            \implies&p\geq N_2(X)\land\forall\,A\in U_p\left(F_2(A,X)U_n=F_3(A,X)U_n\right)\\
            \implies&p\geq N_2(X)\land\forall\,A\in U_p\left(vc\in G_n\right)\\
            \implies&p\geq N_2(X)\land vG_p\subseteq G_n\\
            \implies&p\geq N_2(X)\land p\geq R^n(v);
        \end{aligned}\\
        &\begin{aligned}
            &p\geq N_3(X)\\
            \implies&\forall\,A\in U_p\left(AXU_n=XU_n\right)\\
            \implies&p\geq N_2(X)\land p\geq R^n(v)\land\forall\,A\in U_p\left(F_3(A,X)U_n=XU_n\right)\\
            \implies&p\geq N_2(X)\land p\geq R^n(v)\land\forall\,A\in U_p\left((uv-x)c\in G_n\right)\\
            \implies&p\geq N_2(X)\land p\geq R^n(v)\land (uv-x)G_p\subseteq G_n\\
            \implies&p\geq N_3(X).
        \end{aligned}\label{矩阵等式不动点的等价条件}
    \end{align}

    We are now ready to prove (2)(3)(4).

    From (\ref{矩阵等式不动点的等价条件}) we can see that $U_p\cdot XU_n=\{XU_n\}\iff p\geq N_3(X)$. This proves (2).

    From (\ref{矩阵等式不改变uvwx的等价条件}), for any $X'U_n\in U_{N_2(X)}\cdot XU_n$, we can assume that $u'=u$, $v'=v$, $w'=w$ and $x'=x$, which does not affect the value of $N_3(X')$ by clause (1). Thus $N_3(X')=N_3(X)$. This proves (3).

    From (\ref{矩阵等式不改变uvw的等价条件}), for any $X'U_n\in U_n\cdot XU_n$, we can assume that $u'=u$, $v'=v$ and $w'=w$, which does not change $N_2(X')$ by clause (1). Thus $N_2(X')=N_2(X)$, which shows (4).
\end{proof}

Now we turn to the proof of Lemma~\ref{矩阵乘法和线性变换作用的连续性}. We will use the observations (i) to (iii) that we built in the proof of Lemma~\ref{N2N3的性质}.
\begin{proof}[Proof of Lemma~\ref{矩阵乘法和线性变换作用的连续性}]
    (1) Choose any $a,b\in L$. For any $n\in\omega$, let $p=\max\{n,R^n(a)\}$. Then
    $$(a+L_n)(b+L_p)=ab+aL_p+L_nb+L_nL_p\subseteq ab+L_n.$$

    (2) Choose any $a\in L$ and $z\in G$. For $n\in\omega$, let $p=R^n(a)$. Then
    $$(a+L_n)(z+G_p)=az+aG_p+L_nz+L_nG_p\subseteq az+G_n.$$
\end{proof}


\begin{lemma}\label{U的轨道树节点的秩}
    Let $n\geq 1$ and $T=\otr{U}{U/U_n}$.
    \begin{enumerate}
        \item[(1)] For any $X\in U$, $$\rho_T(N_2(X),U_{N_2(X)}\cdot XU_n)=\max\{0,N_3(X)-N_2(X)-1\};$$
        \item[(2)] for any $X\in U$, if $n<R^n(u)$, then $$\rho_T(N_2(X)-1,U_{N_2(X)-1}\cdot XU_n)=\omega;$$
        \item[(3)] for any $X\in U$, if $n<R^n(u)$, then
        $$\rho_T(n,U_n\cdot XU_n)=\omega+(N_2(X)-n-1);$$
        \item[(4)] for any $X'U_n\in U_{n-1}\cdot XU_n$, $\rho_T(n,U_n\cdot X'U_n)<\omega\cdot 2$;
        \item[(5)] for any $X\in U$, we have $\rho_T(n-1,U_{n-1}\cdot XU_n)=\omega\cdot 2$.
    \end{enumerate}
\end{lemma}
\begin{proof}
    (1) Using Lemma~\ref{N2N3的性质}(2)(3) we can see that for all $X'U_n\in U_{N_2(X)}\cdot XU_n$, $N_3(X)$ is the least $N\in\omega$ such that $U_N\cdot X'U_n=\{X'U_n\}$. Hence
    $$\rho_T(N_2(X),U_{N_2(X)}\cdot XU_n)=\max\{0,N_3(X)-N_2(X)-1\}.$$

    (2) Let $p=N_2(X)-1$. For any $X'U_n\in U_p\cdot XU_n$, we can see that $N_2(X)=N_2(X')$ by Lemma~\ref{N2N3的性质}(4). So by (1) we have $\rho_T(p+1,U_{p+1}\cdot X'U_n)<\omega$. Thus $\rho_T(p,U_p\cdot XU_n)\leq\omega$. On the other hand, we can see that $R^n(uL_p)$ is unbounded in $\omega$ from $n\leq p<R^n(u)$. Therefore, we can find a sequence $(A_k)$ in $U_p$ such that $\lim_{k\to\infty}R^n(ub_k)=\infty$. Then
    \begin{align*}
        N_3(A_kX)\geq R^n(uv+ub_k-x)\to\infty.
    \end{align*}
    This shows $N_3(A_kX)\to\infty$. Now by Lemma~\ref{N2N3的性质}(4) we have $N_2(A_kX)=N_2(X)$. Then by (1) we have
    \begin{align*}
        &\rho_T(N_2(X),U_{N_2(X)}\cdot A_kXU_n)\\
        =&\,\rho_T(N_2(A_kX),U_{N_2(A_kX)}\cdot A_kXU_n)\\
        =&\,\max\{0,N_3(A_kX)-N_2(X)-1\}.
    \end{align*}
    So $\rho_T(p,U_p\cdot XU_n)=\omega$.

    (3) For $X'U_n\in U_n\cdot XU_n$, we have $R^n(u')=R^n(u)$. Using (2) and Lemma~\ref{N2N3的性质}(4) we obtain
    $$\rho_T(N_2(X')-1,U_{N_2(X')-1}\cdot X'U_n)=\rho_T(N_2(X)-1,U_{N_2(X)-1}\cdot X'U_n)=\omega.$$
    This implies $\rho_T(n,U_n\cdot XU_n)=\omega+(N_2(X)-n-1)$.

    (4) By (3), the only remaining case is $n\geq R^n(u')$. But $N_2(X')=n$ in this case. So by (1) we have $\rho_T(n,U_n\cdot X'U_n)<\omega$.

    (5) We can find a sequence $(A_k)$ in $U_{n-1}$, such that $R^n(a_k+u)\to\infty$ and $R^n(b_k+v)=R^n(v)$. Now we have $N_2(A_kX)\to\infty$. Then by (3) we obtain
    \begin{align*}
        &\rho_T(n-1,U_{n-1}\cdot XU_n)\\
        =&\,\sup\{\rho_T(n,U_n\cdot X'U_n)+1:X'U_n\in U_{n-1}\cdot XU_n\}\\
        \geq&\,\sup\{\rho_T(n,U_n\cdot A_kXU_n)+1:k\in\omega\}\\
        =&\,\omega\cdot 2.
    \end{align*}
    Applying (4) we can see that $\rho_T(n-1,U_{n-1}\cdot XU_n)=\omega\cdot 2$.
\end{proof}

\begin{lemma}
    $U$ is proper \hier3.
\end{lemma}

\begin{proof}
    Fix an $n\geq 1$. Then by Lemma~\ref{U的轨道树节点的秩}(5) we have $\rho(\otr{U}{U/U_n})=\omega\cdot 2+n$. Thus $\rho(\mathcal{U})=\omega\cdot 3$. This shows that $U$ is proper \hier3.
\end{proof}

Finally, let
\begin{align*}
    N=\left\{
    \begin{pmatrix}
        1 & 0 & d & f\\
        0 & 1 & 0 & e\\
        0 & 0 & 1 & 0\\
        0 & 0 & 0 & 1
    \end{pmatrix}\in U:d,e,f\in G\right\}.
\end{align*}
Clearly, $N$ is an abelian normal closed subgroup of $U$. It is routine to check that $U/N$ is also an abelian group. So both $N$ and $U/N$ are \hier1. This finishes the construction of our counterexample.

\subsection*{Acknowledgements}
We are grateful to Junhao Chen, Xiangxi Hu, Feng Li and Jie Zou for discussion about some details in Section 7.


\begin{thebibliography}{99}

\bibitem{SAllisonAPana} S. Allison, A. Panagiotopoulos, \textit{The class and dynamics of $\alpha$-balanced Polish groups}. available at \url{https://arxiv.org/abs/2406.06082}, 2024.

\bibitem{BK} H. Becker, A.S. Kechris, The Descriptive Set Theory of Polish Group Actions, Lond. Math. Soc. Lect. Note Ser., vol. 232, Cambridge University Press, 1996.

\bibitem{DingW} L. Ding, X. Wang, \textit{A hierarchy on non-archimedean Polish groups admitting a compatible complete left-invariant metric}, J. Symb. Logic, 1-19. doi:10.1017/jsl.2024.7

\bibitem{gaobook} S. Gao, Invariant Descriptive Set Theory, Monographs and Textbooks in Pure and Applied Mathematics, vol. 293, CRC Press, 2009.

\bibitem{gao98} S. Gao, \textit{On automorphism groups of countable structures}, J. Symb. Logic 63 (1998) 891-896.






\bibitem{malicki11} M. Malicki, \textit{On Polish groups admitting a compatible complete left-invariant metric}, J. Symb. Logic 76 (2011) 437--447.

\bibitem{xuan} M. Xuan, On steinhaus sets, orbit trees and universal properties of various subgroups in the permutation group of natural numbers, Ph.D. thesis, University of North Texas, 2012.

\end{thebibliography}
\end{document}